\documentclass{article}
\title{\textsc{A Path-Complete Approach for Optimal Control of Switched Systems}}

\author{
Léa Ninite$^\star$\thanks{Léa Ninite is a Research Fellow of the Fonds de la Recherche Scientifique -- FNRS. Email: \texttt{lea.ninite@uclouvain.be}}~, Adrien Banse$^\star$, Guillaume O.~Berger$^\star$, Raphaël M.~Jungers$^\star$\\[4pt]
$^\star$ICTEAM, UCLouvain, Belgium
}

\date{}

\usepackage{fullpage}
\usepackage{comment}
\usepackage{mathtools}
\usepackage{circuitikz}
\usepackage{tikz}
\usepackage{subfig}
\usepackage{amsmath}
\usepackage{amssymb}
\usepackage{booktabs}
\usepackage{hyperref}
\hypersetup{
  colorlinks=true,
  linkcolor=black,
  citecolor=black,
  urlcolor=black
}
\usepackage[capitalize]{cleveref}
\usetikzlibrary{arrows.meta}
\usetikzlibrary{positioning}
\usetikzlibrary{calc}
\mathtoolsset{showonlyrefs}

\newtheorem{definition}{Definition}
\newtheorem{theorem}{Theorem}
\newtheorem{proposition}{Proposition}
\newtheorem{lemma}{Lemma}

\newtheorem{remark}{Remark}
\newtheorem{example}{Example}

\newcommand{\Mset}{\langle M\rangle}
\renewcommand{\Re}{\mathbb{R}}
\newcommand{\Ne}{\mathbb{N}}
\newcommand{\Exp}{\mathbb{E}}
\newenvironment{proof}[1][Proof]{%
  \noindent\textit{#1. }%
}{\hfill $\square$\\ \par}
\DeclareMathOperator*{\argmax}{arg\,max}
\DeclareMathOperator*{\argmin}{arg\,min}
\newcommand{\calS}{\mathcal{S}}
\newcommand{\calF}{\mathcal{F}}
\newcommand{\calE}{\mathcal{E}}
\newcommand{\calG}{\mathcal{G}}
\newcommand{\calT}{\mathcal{T}}
\newcommand{\calH}{\mathcal{H}}
\newcommand{\nxn}{{n\times n}}
\newcommand{\nxm}{{n\times m}}
\newcommand{\mxm}{{m\times m}}
\newcommand{\mxn}{{m\times n}}
\newcommand{\txtquad}{\mathrm{quad}}

\newcommand{\ct}{\tilde{c}}

\begin{document}
\maketitle
\begin{abstract}
We study the problem of estimating the \emph{value function} of discrete-time switched systems under \emph{arbitrary} switching. Unlike the switched LQR problem, where both inputs and mode sequences are optimized, we consider the case where switching is exogenous. For such systems, the number of possible mode sequences grows exponentially with time, making the exact computation of the value function intractable. This motivates the development of tractable bounds that approximate it. We propose a novel framework, based on \emph{path-complete} graphs, for constructing computable upper bounds on the value function. In this framework, multiple quadratic functions are combined through a directed graph that encodes dynamic programming inequalities, yielding convex and sound formulations. For example, for switched linear systems with quadratic cost, we derive tractable LMI-based formulations and provide computational complexity bounds. We further establish approximation guarantees for the upper bounds and show asymptotic non-conservativeness using concepts from graph theory. Finally, we extend the approach to controller synthesis for systems with affine control inputs and demonstrate its effectiveness on numerical examples.

\end{abstract}
\section{Introduction}
Discrete-time switched linear systems are multi-modal discrete-time systems wherein each mode $i$ corresponds to a linear transition system of the form
\begin{equation}
    x_{k+1}=A_ix_k+B_iu_k. 
\end{equation}
These systems appear naturally in a wide range of applications---from mechanical systems~\cite{blanchini2012constant} to power systems~\cite{sanchez2019practical} and cyber-physical systems \cite{donkers2011stability} or as abstractions of more complex nonlinear or hybrid systems~\cite{liberzon2003switching,jungers2009joint,sun2011stability}.
In this work, we are interested in the cost-to-go analysis and optimization for these systems \emph{under arbitrary switching}.
This means that given a cost function $c(x, u)$ mapping state--input pairs to cost values and an initial point $x_0$, we aim to evaluate or optimize the \emph{worst-case total cost} of the system from $x_0$, that is, the largest total cost among all trajectories (i.e., for all mode sequences) starting from $x_0$.
This problem is important in safety- or energy-critical applications because it provides guarantees of safety or performance, even under adversarial conditions, such as in~\cite{liberzon2003switching,sun2011stability,blanchini2008set} for stability.\\\\
For non-switched linear systems, in the well-known \emph{LQR} setting, the cost is given by
\begin{equation}
    c(x, u) = x^\top Q x + u^\top R u, \quad Q\succeq 0,\: R\succ 0.
\end{equation}
In this case, the optimal cost-to-go (called hereafter the \emph{value function} in the controlled case) admits a closed-form quadratic expression, which can be efficiently computed through the solution of the Riccati equation~\cite{mehrmann1991autonomous}. Similarly, if the system has no input, the associated cost-to-go (called hereafter the \emph{value function} in the autonomous case) can also be efficiently computed by solving a linear Bellman equation~\cite{bertsekas2012dynamic}.
However, for switched systems, the situation is different.
When the switching is arbitrary, the number of possible sequences grows exponentially with time, making the computation of the value function intractable, in both the controlled and the autonomous case, even for a given initial state $x_0$.
In particular, there is in general no closed-form expression for the value function, which is typically not quadratic and can be non-smooth~\cite{zhang2009value}.
This motivates the development of tractable approximations of the value function such as provable upper and lower bounds.\\\\
In this work, we propose a \emph{path-complete approach} to construct tractable bounds on the value function of switched linear systems under arbitrary switching.
The path-complete framework, first introduced in~\cite{ahmadi2014joint} for the stability analysis of switched linear systems, generalizes the use of a \emph{single} Lyapunov function to \emph{multiple} Lyapunov functions, which combined together in a combinatorial way lead to tighter stability certificates. 
In this work, we adopt the same combinatorial structure to obtain upper bounds on the value function of arbitrarily switched systems. 
Specifically, the multiple functions are computed by encoding dynamic programming conditions on a directed graph: each node $\alpha$ is associated with a function $V_\alpha$ (to be determined), and each edge $\alpha\rightarrow \beta$, labeled with a system mode $i$, represents a dynamic programming inequality $V_{\alpha} \geq c + V_{\beta}\circ A_i$ on these functions. 
A key advantage of this approach is that, in some situations (e.g., when considering a quadratic template for the functions $V_\alpha$), it provides a sound and tractable way (e.g., through semidefinite programming) of upper bounding the value function of switched linear systems under arbitrary switching, extending existing Lyapunov-based approaches.\\\\
This work focuses mostly on bounding the value function of \emph{autonomous} switched linear systems under arbitrary switching.
We note that the autonomous case is relevant not only for safety or performance analysis of closed-loop systems but also for safe and efficient controller synthesis (e.g., through approximate policy improvement~\cite{munos2003error} or informed-search algorithms like A$^*$~\cite{hart1968formal}).
Next to this, we also consider the \emph{controlled} case and show that the path-complete framework can be used to directly synthesize a controller that minimizes an upper bound on the value function of the associated closed-loop system.
\paragraph{Outline} Concretely, our contributions are as follows.
After introducing the problem (Section~\ref{sec:preliminaries}) and proving general dynamic programming bounds on the value function of autonomous switched nonlinear systems under arbitrary switching (Section~\ref{sec:upper-bound-general}), we provide the first path-complete framework for upper bounding the value function of such systems (Section~\ref{sec:path-complete-autonomous}).
Then, we apply this framework to switched linear systems with quadratic cost (Section~\ref{sec:autonomous-linear}), deriving tractable LMI formulations, along with computational complexity bounds.
We also derive approximation guarantees that quantify the tightness of the computed upper bound on the value function, and prove that the proposed approach is asymptotically non-conservative when using the duals of the so-called \emph{De Bruijn} graphs~\cite{de1948combinatorial} (see~\Cref{def:debruijn} below) with order tending to infinity.
Then, we apply the path-complete approach to the \emph{control} of switched linear systems with affine control input under arbitrary switching, providing tractable LMI formulations to controller design along with upper bounds on the value function (Section~\ref{sec:path-complete-control}).
Finally, we demonstrate the usefulness and efficiency of our approach on numerical examples (Section~\ref{sec:experiments}).
\paragraph{Related work} The \emph{switched LQR problem}, which aims to minimize the total quadratic cost of switched linear systems by choosing the affine control input \emph{and} the mode sequence, has received much attention in the literature.
For instance,~\cite{rantzer2005approximate,gorges2011optimal,zhang2012infinite,antunes2017linear,hou2024improved} propose efficient techniques to approximate the optimal cost-to-go under input-and-mode control, and~\cite{wu2020optimal} proposes a polynomial-time exact computation method under additional assumptions on the system. These techniques cannot be applied to our problem, as they treat the mode as a control variable, whereas in our setting the switching is arbitrary. Nevertheless, the work \cite{rantzer2005approximate} mentioned above bears several interesting connections with our approach. There, the author considers multiple quadratic functions satisfying inequalities of the form $V_i \leq c + V_j \circ A_{ij}$ (where $i,j$ are modes) to obtain \emph{lower} bounds on the optimal value function in the switched LQR setting, whereas we consider inequalities of the form 
$V_\alpha \geq c + V_\beta \circ A_i$ (where $\alpha,\beta$ are nodes and $i$ a mode) to derive \emph{upper} bounds on the value function. Despite the apparent similarity, there are several key differences with our approach. First, the quadratic functions $V_i$ are attached to modes of the system---and not to \emph{abstract} nodes as in our approach.
Hence, the graph that gives the inequalities is fixed and cannot be used as an algorithm parameter to balance efficiency and conservativeness.
By contrast, our approach considers arbitrary path-complete graphs (in particular, nodes do not necessarily correspond to system modes), allowing to augment the ``size'' of the graph to obtain tighter bounds.
Namely, we show that our approach is asymptotically non-conservative when considering a particular type of graphs (dual De Bruijn graphs) whose size tends to infinity; a similar result is not present in~\cite{rantzer2005approximate}.
Second, while~\cite{rantzer2005approximate} also derives approximation guarantees by scaling their lower bound, their approach is more conservative than ours.\footnote{In \cite{rantzer2005approximate}, the author uses a convex over-approximation of the minimum of quadratic functions consisting in any convex combination of these functions, whereas we leverage the less conservative \emph{S-procedure} from~\cite{boyd2004convex} (see~\cref{sec:approximation-guarantees} for details).}\\\\
The problem of optimal control in the context of nondeterministic systems has also received attention; for instance, in nondeterministic dynamic programming~\cite{fujita2004nondeterministic}, robust MDPs~\cite{suilen2024robust, Nilim2005RobustCO} or two-player games~\cite{moos2022robust}.
Nevertheless, to the best of our knowledge, no previous work on the value function of switched linear systems under arbitrary switching is available in the literature.

\paragraph{Notations} Given $M\in \mathbb N$, we let $\Mset\coloneqq\{1,\dots,M\}$.
The set of $n\times n$ symmetric positive definite (semi-definite) matrices is denoted by $\Re^\nxn_{\succ0}$ ($\Re^\nxn_{\succeq0}$).
The set of nonnegative-valued functions defined on $\Re^n$ (i.e., functions $f:\Re^n\to\Re_{\geq0}$) is denoted by $\calF^n_{\geq0}$.

\section{Problem statement}\label{sec:preliminaries}
\subsection{Switched systems and value function}

We consider a discrete-time autonomous switched system of the form
\begin{equation}
x_{k+1} = f_{\sigma(k)} (x_k), \quad k\in\Ne,
\label{eq:switched_controlled}
\end{equation}
where $x_k\in\Re^n$ is the \emph{state} at time $k$, $\sigma(k)\in\Mset$ is the \emph{mode} at time $k$ and for all $i\in\Mset$, $f_i:\Re^n\to\Re^n$.
The function $\sigma : \Ne \to \Mset$ is called the \emph{switching signal}.\\\\
Given $x\in\Re^n$ and a switching signal $\sigma:\Ne\to\Mset$, we denote by $\xi(\cdot,x,\sigma)$ the solution of~\eqref{eq:switched_controlled} with switching signal $\sigma$, i.e., $\xi(k,x,\sigma)=x_k$ for all $k\in\Ne$ where $x_0=x$ and $\{x_k\}_{k=0}^\infty$ satisfies~\eqref{eq:switched_controlled}.\\\\
We also consider a cost function $c:\Re^n\to\Re_{\geq0}$, mapping states to nonnegative cost values.\\\\
In this work, the switching signal is considered as an external \emph{uncontrolled} signal, meaning that the analysis of the system~\eqref{eq:switched_controlled} is made using the worst-case framework.
In particular, we consider the \emph{worst-case cost-to-go} defined as follows:

\begin{definition}[Cost-to-go and value function]
Consider the switched system~\eqref{eq:switched_controlled} and a cost function 
$c:\Re^n \to \Re_{\geq0}$. 
For each switching signal $\sigma:\Ne \to \Mset$, the \emph{cost-to-go} is the mapping 
$J_\sigma:\Re^n \to [0,+\infty]$ defined by
\[
J_\sigma(x) \coloneqq \sum_{k=0}^{\infty} c(\xi(k,x,\sigma)),
\]
with the convention that $J_\sigma(x)=+\infty$ if the series diverges.\\
The \emph{value function} $J:\Re^n \to [0,+\infty]$ captures the worst-case cost over all switching signals and is defined by
\[
J(x) \coloneqq \sup_{\sigma:\Ne\to\Mset} J_\sigma(x).
\]\vskip0pt
\label{definition:cost_to_go_val_function}
\end{definition}

The value function of switched linear systems can be very challenging to compute exactly, even with a quadratic cost function.\footnote{The paper~\cite{zhang2009value} shows that the \emph{finite-horizon} value function of switched linear systems under controlled switching with a quadratic cost function is a piecewise quadratic function whose number of pieces can grow exponentially with the horizon.
The infinite-horizon value function is therefore in general not quadratic or smooth and difficult to approximate.
A similar argument can be applied to systems under \emph{arbitrary} switching, although the proof is omitted due to space limitations.}
Hence, the objective of this paper is to derive tractable \emph{bounds} on the value function. 
To this end, we rely on the path-complete framework (introduced in~\cref{subsection: path-complete graphs} below), which enables the computation of upper bounds in a tractable manner as combinations (min or max) of quadratic functions.

\subsection{Path-complete graphs}\label{subsection: path-complete graphs}

The path-complete Lyapunov framework, introduced in~\cite{ahmadi2014joint}, generalizes classical quadratic Lyapunov techniques for analyzing the stability of discrete-time switched linear systems.
At its core, is the notion of \emph{path-complete graph}, which encodes Lyapunov inequalities between quadratic functions along different modes of the system, with the property that any mode sequence of the system can be obtained through a path in the graph.
In what follows, we briefly recall the definitions that we need.\\\\
We consider a \emph{directed labeled graph} $\calG = (\calS, \calE)$, where $\calS$ is a finite set of nodes and $\calE \subseteq \calS \times \calS\times\Mset$ is the set of edges labeled by elements of $\Mset$.
Each edge $(\alpha, \beta, i)\in\calE$ represents a possible transition from node $\alpha$ to node $\beta$ under mode $i$ of the switched system.

\begin{definition}[Path-complete graph]
Let $\calG = (\calS,\calE)$ be a directed labeled graph.
We say that $\calG$ is \emph{path-complete} (for $\Mset$) if for any $l\in\Ne_{>0}$ and any sequence $\sigma = (i_1,\ldots,i_l)\in\Mset^l$, there exists a path 
$\{(\alpha_k,\alpha_{k+1}, i_k)\}_{k=1}^l$ such that for each $k\in\{1,\ldots,l+1\}$, $\alpha_k\in\calS$, and for each $k\in\{1,\ldots,l\}$, $(\alpha_k,\alpha_{k+1},i_k)\in\calE$.
\end{definition}

An example of a path-complete graph is shown in \cref{fig:pc-graph}. The graph in~\cref{fig:nonpc-graph} is not path-complete (for example, it cannot generate the sequence $22$).
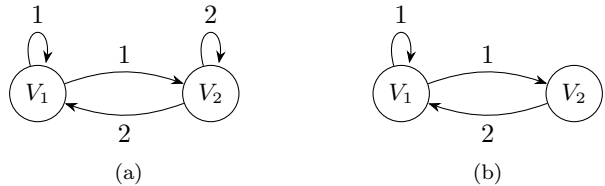
\begin{figure}[ht]
\centering

\subfloat[]{
\begin{tikzpicture}[>=Stealth, scale=1, transform shape,
    node distance=2.3cm,
    vertex/.style={circle, draw, minimum size=0.7cm, font=\small}]
    
    \node[vertex] (V1) {$V_1$};
    \node[vertex, right of=V1] (V2) {$V_2$};
    \path[->] (V1) edge[loop above] node{$1$} (V1);
    \path[->] (V2) edge[loop above] node{$2$} (V2);
    \path[->] (V1) edge[bend left=20] node[above]{$1$} (V2);
    \path[->] (V2) edge[bend left=20] node[below]{$2$} (V1);
\end{tikzpicture} \label{fig:pc-graph}
}
\hspace{1.3cm}
\subfloat[]{
\begin{tikzpicture}[>=Stealth, scale=1, transform shape,
    node distance=2.3cm,
    vertex/.style={circle, draw, minimum size=0.7cm, font=\small}]
    
    \node[vertex] (V1) {$V_1$};
    \node[vertex, right of=V1] (V2) {$V_2$};
    \path[->] (V1) edge[loop above] node{$1$} (U1);
    \path[->] (V1) edge[bend left=20] node[above]{$1$} (V2);
    \path[->] (V2) edge[bend left=20] node[below]{$2$} (V1);
\end{tikzpicture}  \label{fig:nonpc-graph}}

\caption{(a) Path-complete graph with two nodes, for a system with two switching modes. (b) Graph not path-complete.}
\label{fig:pc-vs-copc}
\end{figure}

Given a graph $\calG = (\calS,\calE)$, its \emph{dual graph} $\calG^{\mathsf{T}} = (\calS, \calE^{\mathsf{T}})$ is obtained by reversing the direction of each edge, i.e., $(\alpha, \beta, i)\in\calE \Leftrightarrow (\beta, \alpha, i)\in\calE^{\mathsf{T}}$. A graph is path-complete if and only if its dual is path-complete. \\\\
Below, we define two special classes of path-complete graphs:
\begin{definition}[Complete and co-complete graphs]
A directed labeled graph $\calG = (\calS, \calE)$ is \emph{complete} (for $\Mset$) if for each node $\alpha\in\calS$ and each label $i\in\Mset$, there exists at least one node $\beta\in\calS$ such that $(\alpha, \beta, i)\in\calE$. Conversely, $\calG$ is \emph{co-complete} if for each node $\beta\in\calS$ and each label $i\in\left <M \right >$, there exists at least one node $\alpha\in\calS$ such that $(\alpha, \beta, i)\in\calE$.\label{definition:complete_cocomplete}
\end{definition}

The graph in \cref{fig:pc-graph} is co-complete.

\section{Dynamic programming inequalities and bounds on the value function}\label{sec:upper-bound-general}

In stochastic optimal control, i.e., when one aims to minimize the expected total cost $\bar{J}$ of trajectories of a stochastic system of the form $x_{k+1}=f(x_k,v_k)$ where $\{v_k\}_{k=0}^\infty$ is an i.i.d.~noise sequence, one has the famous \emph{Bellman equation}~\cite{bertsekas2012dynamic}: $\bar{J}(x)=c(x)+\Exp_v\bar{J}(f(x,v))$.
In the case of nondeterministic and nonstochastic systems\footnote{The system~\eqref{eq:switched_controlled} is nondeterministic because the switching sequence is arbitrary, and nonstochastic because no probability distribution is associated with it.}, a similar equation can be derived (see, e.g.,~\cite{fujita2004nondeterministic,suilen2024robust}).
We recall it here for convenience (although we will only focus on its inequality versions in Propositions~\ref{prop:bellman-upper} and~\ref{prop:bellman-lower}):

\begin{proposition}
The value function $J$ of system~\eqref{eq:switched_controlled} with cost $c:\Re^n\to\Re_{\geq0}$ satisfies the following \emph{dynamic programming equation}:
\begin{equation}
J(x)=c(x)+\max_{i\in\Mset} J(f_i(x)), \quad \forall\,x\in\Re^n.\label{eq:bellman_eq}
\end{equation}
\vskip0pt
\end{proposition}

In practice, when one searches for a value function in a limited template (e.g., quadratic functions), it may be impossible to find a function $J$ that satisfies~\eqref{eq:bellman_eq} exactly.
For this reason, we consider \emph{relaxations} of this dynamic programming equation, 
expressed as \emph{dynamic programming inequalities}. Depending on the inequality direction, 
these relaxations yield either upper or lower bounds on the exact value function.\footnote{We note that Bellman inequalities were also used in stochastic optimal control to obtain bounds on the expected total cost $\bar{J}$; see, e.g.,~\cite{lu2021convex,rantzer2005approximate,lincoln2006relaxing}.}

The following two propositions formalize this.

\begin{proposition}[Upper bound]\label{prop:bellman-upper}
Consider system~\eqref{eq:switched_controlled} and a cost function $c:\Re^n\to\Re_{\geq0}$.
Let $V:\Re^n\to\Re_{\geq0}$ satisfy the following \emph{dynamic programming inequality}:
\begin{equation}\label{eq:bellman-upper}
V(x)\geq c(x) + \max_{i\in\Mset} V(f_i(x)), \quad \forall\,x\in\Re^n.
\end{equation}
Then, for all $x\in\Re^n$, $V(x)\geq J(x)$, where $J$ is the value function of system~\eqref{eq:switched_controlled} with cost $c$.
\end{proposition}

\begin{proof}
Let $x\in\Re^n$ and $\sigma:\Ne\to\Mset$.
For each $k\in\Ne$, denote $x_k=\xi(k,x,\sigma)$.
Note that by~\eqref{eq:bellman-upper}, it holds that for all $k\in\Ne$, $V(x_k)\geq c(x_k)+V(x_{k+1})$.
Hence, for all $H\in\Ne_{>0}$,
\[
V(x)\geq V(x)-V(x_H) = \sum_{k=0}^{H-1} V(x_k)-V(x_{k+1}) \geq \sum_{k=0}^{H-1} c(x_k).
\]
Taking the limit when $H\to\infty$, we get that $V(x)\geq J_\sigma(x)$.
Then, taking the supremum over $\sigma$, we obtain that $V(x)\geq J(x)$, concluding the proof. 
\end{proof}

\begin{proposition}[Lower bound]\label{prop:bellman-lower}
Consider system~\eqref{eq:switched_controlled} and a cost function $c:\Re^n\to\Re_{\geq0}$.
Let $W:\Re^n\to\Re_{\geq0}$ satisfy the following \emph{dynamic programming inequality}:
\begin{equation}\label{eq:bellman-lower}
W(x)\leq c(x) + \max_{i\in\Mset} W(f_i(x)), \quad \forall\,x\in\Re^n.
\end{equation}
Assume that $W$ is continuous at $0$, $W(0)=0$, and all trajectories of the system converge to $0$ (i.e., the system is asymptotically stable).
Then, for all $x\in\Re^n$, $W(x)\leq J(x)$, where $J$ is the value function of system~\eqref{eq:switched_controlled} with cost $c$.
\end{proposition}

\begin{proof}
Let $x\in\Re^n$.
Let $(x_k)_{k=0}^\infty$ and $(i_k)_{k=0}^\infty$ be defined recursively by $x_0=x$ and for all $k\in\Ne$, $i_k\in\argmax_{i\in\Mset}W(f_i(x_k))$ and $x_{k+1}=f_{i_k}(x_k)$.
Note that by~\eqref{eq:bellman-lower}, it holds that for all $k\in\Ne$, $W(x_k)\leq c(x_k)+W(x_{k+1})$.
Hence, for all $H\in\Ne_{>0}$,
\[
W(x)-W(x_H) = \sum_{k=0}^{H-1} W(x_k)-W(x_{k+1}) \leq \sum_{k=0}^{H-1} c(x_k).
\]
Taking the limit when $H\to\infty$ (and using that $x_H\to0$), we get that $W(x)\leq J_\sigma(x)$, where $\sigma:\Ne\to\Mset$ is defined by $\sigma(k)=i_k$.
This implies that $W(x)\leq J(x)$, concluding the proof. 
\end{proof}
This paper focuses on computing upper bounds on the value function.
We employ the dynamic programming inequality of~\cref{prop:bellman-upper} together with the path-complete framework to obtain tractable approximations, represented as combinations of simple functions such as quadratics. Using Proposition~\ref{prop:bellman-lower}, we further establish approximation guarantees, showing that a scaled version of the upper bound serves as a lower bound.

\section{Path-complete upper bound on the value function}\label{sec:path-complete-autonomous}

The path-complete framework consists in combining several functions in a given template of functions (e.g., quadratic functions) in order to compute a function satisfying the dynamic programming inequality in Proposition~\ref{prop:bellman-upper}, thereby providing an upper bound on the value function.\\\\
More precisely, let us consider a path-complete graph $\calG=(\calS,\calE)$ (for $\Mset$) and a template of functions $\calT\subseteq\calF^n_{\geq0}$. Each node $\alpha$ of this graph is associated with a function $V_\alpha\in\calT$. For system~\eqref{eq:switched_controlled} and a cost function $c:\Re^n \to \Re_{\ge0}$, the graph encodes the following inequalities, inspired by~\eqref{eq:bellman-upper}:
\begin{equation}
V_\alpha(x) \ge c(x) + V_\beta(f_i(x)), \quad \forall\,x\in\Re^n,\:\forall\,(\alpha,\beta,i)\in\calE.
\label{eq:ineq_V}
\end{equation}\vskip0pt

\begin{example}
As an illustration, consider the path-complete graph shown in \cref{fig:pc-graph}.
The inequalities associated with its edges are:
\begin{itemize}
    \item Transition $V_1 \to V_1$: $V_1(x) \geq c(x) + V_1(f_1(x))$, $\forall\,x\in\Re^n$,
    \item Transition $V_2 \to V_2$: $V_2(x) \geq c(x) + V_2(f_2(x))$, $\forall\,x\in\Re^n$,
    \item Transition $V_1 \to V_2$: $V_1(x) \geq c(x) + V_2(f_1(x))$, $\forall\,x\in\Re^n$,
    \item Transition $V_2 \to V_1$: $V_2(x) \geq c(x) + V_1(f_2(x))$, $\forall\, x\in\Re^n$.~\hfill$\triangleleft$
\end{itemize}
\end{example}

Given a collection of functions $\{V_\alpha\}_{\alpha\in\calS}\subseteq \calT$ satisfying~\eqref{eq:ineq_V}, 
we aim to construct an upper bound $V$ for the true value function $J$, that is, a function $V$ such that $V(x) \ge J(x)$ for each $x\in \Re^n$, by appropriately combining the functions $V_\alpha$. The specific construction depends on the structure of the underlying graph $\calG=(\calS,\calE)$:
\begin{itemize}
    \item If $\calG$ is \emph{complete}, taking the pointwise \emph{minimum} over the functions $V_\alpha$ yields a valid upper bound.
    \item If $\calG$ is \emph{co-complete}, taking the pointwise \emph{maximum} over the functions $V_\alpha$ yields a valid upper bound.
\end{itemize}
These statements are formalized in the following theorem and are based on previous work extracting common Lyapunov functions from path-complete Lyapunov functions~\cite[Corollaries~3.4 and~3.5]{ahmadi2014joint}.

\begin{theorem}\label{theorem:ub_complete_cocomplete}
Consider system~\eqref{eq:switched_controlled}, a cost function $c:\Re^n\to\Re_{\geq0}$, 
a directed labeled graph $\calG=(\calS,\calE)$ and a template $\calT\subseteq\calF^n_{\geq0}$.
Assume that $\{V_\alpha\}_{\alpha\in\calS}\subseteq\calT$ satisfy~\eqref{eq:ineq_V}.  
Define $V:\Re^n\to\Re_{\ge0}$ by
\begin{subequations} \label{eq:Vcases}
\begin{align}
V(x) &\coloneqq \min_{\alpha\in\calS} V_\alpha(x), \quad \text{if $\calG$ is complete for $\Mset$,} \label{eq:def_V_bar_complete} \\
V(x) &\coloneqq \max_{\alpha\in\calS} V_\alpha(x), \quad \text{if $\calG$ is co-complete for $\Mset$.} \label{eq:def_V_bar_cocomplete}
\end{align}
\end{subequations}
Then, for all $x\in\Re^n$, $V(x)\ge J(x)$, where $J$ is the value function of system~\eqref{eq:switched_controlled} with cost~$c$.
\label{prop:upper_bound_general}
\end{theorem}

\begin{proof}
For both the complete and the co-complete cases, we will prove that
\begin{equation}
    V(x)\geq c(x)+\max_{i\in\Mset}V(f_i(x)),\quad \forall\,x\in\Re^n,
    \label{eq:proof_upper_bound2}
\end{equation}  
which, by \cref{prop:bellman-upper}, implies that $V(x)\ge J(x)$ for all $x\in \Re^n$.

\medskip\noindent\textbf{Case 1: $\calG$ is complete.}  
We start with the case of a complete graph, i.e., $V(x)=\min_{\alpha\in\calS}V_\alpha(x)$.
Define the set
\[
E\coloneqq\{(\alpha,i)\in\calS\times\Mset:\ \exists\,\beta \in\calS \text{ such that } (\alpha, \beta,i)\in \calE\}.
\]
From~\eqref{eq:ineq_V}, we deduce that
\begin{equation}
    V_\alpha(x) \ge c(x) + \min_{\beta\in\calS} V_\beta(f_i(x)),\quad 
    \forall\,x\in\Re^n,\ \forall\,(\alpha,i)\in E. 
\end{equation}  
By definition of $V$, 
\begin{equation}
    V_\alpha(x) \ge c(x) + V(f_i(x)),\quad 
    \forall\,x\in\Re^n,\ \forall\,(\alpha,i)\in E.
    \label{eq:proof_upper_bound1}
\end{equation}
Since the graph is complete, it follows from \cref{definition:complete_cocomplete} that $E=\calS\times\Mset$.
Therefore,~\eqref{eq:proof_upper_bound1} can be rewritten as
\begin{equation}
    V_\alpha(x) \ge c(x) + V(f_i(x)),\quad 
    \forall\,x\in\Re^n,\ \forall\,\alpha\in\calS,\ \forall\,i\in\Mset.
\end{equation}
Taking the minimum over $\alpha$ and the maximum over $i$, we obtain~\eqref{eq:proof_upper_bound2}.

\medskip\noindent\textbf{Case 2: $\calG$ is co-complete.}  
We now consider the co-complete case, i.e., $V(x)=\max_{\alpha \in \calS} V_\alpha(x)$.
The reasoning is similar to the complete case.
Define the set
\[
E \coloneqq \{(\beta,i)\in\calS\times\Mset : \exists\,\alpha\in\calS \text{ such that } (\alpha, \beta,i)\in \calE\}.
\]
From~\eqref{eq:ineq_V}, we deduce that
\begin{equation}
   \max_{\alpha\in \calS} V_\alpha(x) \ge c(x) +  V_\beta(f_i(x)),\quad 
   \forall\,x\in\Re^n,\ \forall\,(\beta,i)\in E. 
\end{equation}
By definition of $V$,
\begin{equation}
    V(x)\geq c(x)+V_\beta(f_i(x)),\quad 
    \forall\,x\in\Re^n,\ \forall\,(\beta,i)\in E.
    \label{eq:proof_upper_bound_cocomplete}
\end{equation}
Since the graph is co-complete, it follows from \cref{definition:complete_cocomplete} that $E=\calS\times\Mset$.
Therefore,~\eqref{eq:proof_upper_bound_cocomplete} can be rewritten as
\begin{equation}
    V(x)\geq c(x)+V_\beta(f_i(x)),\quad 
    \forall\,x\in\Re^n,\ \forall\,\beta\in\calS,\ \forall\,i\in\Mset. 
\end{equation}
Taking the maximum over $\beta$ and then over $i$, we obtain~\eqref{eq:proof_upper_bound2}, which concludes the proof.
\end{proof}

\begin{remark}
Note that, in the same way as in \cite{angeli2017}, it is possible to generalize~\cref{theorem:ub_complete_cocomplete} to any path-complete graph, in which case the upper bound takes the form  
\begin{equation}
V(x) = \min_{\mathcal A_1, \dots, \mathcal A_k  \subseteq \calS} \max_{\alpha\in\mathcal A_i} V_{\alpha}(x), 
\end{equation}
where $\mathcal A_1, \ldots, \mathcal A_k$ are nodes of the complete sub-graph of the \emph{observer graph} (see \cite[Theorem~1]{angeli2017}). We skip this result for the sake of clarity and conciseness.\hfill$\triangleleft$
\end{remark}

\section{Results for autonomous switched linear systems}\label{sec:autonomous-linear}

In this section, we apply the framework developed in the previous section to switched linear systems with quadratic cost.
We show that it allows to obtain \emph{tractable} upper bounds on the value function through LMI formulations of the path-complete inequalities~\eqref{eq:ineq_V}.
We further show that the framework is asymptotically non-conservative (in the sense that it can yield upper bounds that converge to the true value function) when the graph is the so-called dual De~Bruijn graph (see~\Cref{def:debruijn} below) whose order tends to infinity.
Finally, we derive an approximation guarantee by appropriately scaling the previously computed upper bound. The scaled function is guaranteed to be a valid lower bound, and the scaling factor is obtained by solving a set of LMIs.\\\\
Concretely, we consider discrete-time autonomous switched \emph{linear} systems of the form
\begin{equation}
x_{k+1} = A_{\sigma(k)} x_k,\quad k\in\Ne,
\label{eq:linear_autonomous}
\end{equation}
where for each $i\in\Mset$, $A_i\in\Re^\nxn$.
Hence,~\eqref{eq:linear_autonomous} is a special case of the system in~\eqref{eq:switched_controlled} where each map $f_i$ is the linear map $x\mapsto A_ix$. We also assume that the cost is quadratic, namely, that $c(x)=x^\top Qx$ for some $Q\in\Re^\nxn_{\succ0}$.
Finally, we consider the template of quadratic functions, that is, $\calT_\txtquad=\{x^\top Px : P\in\Re^\nxn_{\succeq0}\}$.

\subsection{Path-complete LMIs and computational complexity}\label{section:LMIs_formulation_linear_auton_syst}

Under the assumption that $V_\alpha\in\calT_\txtquad$ for each $\alpha\in\calS$, we parametrize them as
\begin{equation}
    V_\alpha(x) = x^\top P_\alpha x, \quad P_\alpha\in\Re^\nxn_{\succeq0},\quad \forall\,\alpha\in\calS.
\end{equation}
Under this parametrization, the path-complete inequalities~\eqref{eq:ineq_V}
can be equivalently written as the following set of LMIs:
\begin{equation}\label{eq:ineq_V_lmi}
    P_\alpha \succeq Q + A_i^\top P_\beta A_i^{}, 
    \quad \forall\,(\alpha,\beta,i)\in\calE.
\end{equation}\vskip0pt

\begin{remark} \label{rem:objective}
Note that any collection of matrices $\{P_\alpha\}_{\alpha\in\calS}\subseteq \Re^\nxn_{\succeq0}$ satisfying~\eqref{eq:ineq_V_lmi} provides an upper bound on the true value function $J$, obtained for each $x$ by taking the minimum or maximum of $\{x^\top P_\alpha x\}_{\alpha\in\calS}$ in the case of a complete or co-complete graph, respectively (see~\cref{theorem:ub_complete_cocomplete}). However, this does not guarantee that the bound is tight. It is therefore crucial to choose an objective function that penalizes large values of the upper bound. For example: 
\begin{itemize}
    \item If one wants to minimize the value of the upper bound at a given $x_0$, one can use the following objective:
    \begin{equation}
        \begin{array}{ll}
        \displaystyle \min_{\{P_\alpha\}_{\alpha \in \calS}} \:\min_{\alpha\in\calS}\: x_0^\top P_\alpha x_0^{} \quad &\text{if $\calG$ is complete}, \label{eq:pointwise-complete} \\[2pt]
        \displaystyle \min_{\{P_\alpha\}_{\alpha \in \calS}}  \:\max_{\alpha\in\calS}\: x_0^\top P_\alpha x_0^{} \quad &\text{if $\calG$ is co-complete}.
        \end{array}
    \end{equation}
    This guarantees that the upper bound is minimal at $x_0$ (among all upper bounds derived from matrices satisfying~\eqref{eq:ineq_V_lmi}).\footnote{Note that the first objective function in~\eqref{eq:pointwise-complete} is not convex in $\{P_\alpha\}_{\alpha\in\calS}$.
    However, it can be easily convexified by considering $\lvert\calS\rvert$ subproblems (one with objective function $x_0^\top P_\alpha x_0^{}$ for each $\alpha\in\calS$, see Appendix~\ref{app:dcdc} for a specific example).}
    \item On the contrary, if one wants that the upper bound is ``globally'' small, i.e., for a large number of values of $x$, one can use the objective: 
    \begin{equation}
       \min_{\{P_\alpha\}_{\alpha \in \calS}} \: \sum_{\alpha\in\mathcal{S}} \ell(P_\alpha),
       \label{eq:obj_log_det}
    \end{equation}
    where $\ell(P)$ is any function that penalizes ``large'' values of $P$,\footnote{For instance, $\ell(P)=\log(\det(P))$, $\ell(P)=\mathrm{trace}(P)$ or $\ell(P)=\lVert P\rVert^2$, which can all be formulated as semidefinite programs.} so as to yield small values of the upper bound on the whole state space.
\end{itemize}
Note that other objective functions can be considered, for instance if one wants to minimize the expected value of the upper bound (for some given state distribution).
We leave the study of other objective functions to further work.\hfill$\triangleleft$
\end{remark}
As an illustrative example, we apply the path-complete approach to a simple switched linear system.
\begin{example}
Consider system~\eqref{eq:linear_autonomous} with
\begin{equation}
A_1 = \frac{1}{1.75}
\begin{bmatrix}
1.3 & 0 \\
1 & 0.3
\end{bmatrix}, 
\quad
A_2 = \frac{1}{1.75}
\begin{bmatrix}
-0.3 & 1 \\
0 & -1.3
\end{bmatrix}.\label{eq:matrices_simple_ex}
\end{equation}
Taking the objective function in~\eqref{eq:obj_log_det} with $\ell=\mathrm{trace}$, using $Q=I$, and considering the LMIs~\eqref{eq:ineq_V_lmi} on the co-complete graph depicted in~\cref{fig:pc-graph}, we obtain the following solution:
\[
P_1 = 
\begin{bmatrix}
3.32 & 0.14 \\
0.14 & 1.14
\end{bmatrix}, \quad
P_2 = 
\begin{bmatrix}
1.14 & -0.14 \\
-0.14 & 3.32
\end{bmatrix}.
\]
The upper bound $V(x)=\max_{\alpha\in\{1,2\}} x^\top P_\alpha x$ is shown in \cref{fig:illustrative_example} against the true value function (given for reference only and approximated by truncating the infinite-horizon sum at a horizon \(H\) such that the cumulative cost beyond \(H\) is negligible).~\hfill$\triangleleft$
\end{example}

\begin{figure}[h]
\centering
\includegraphics[width=0.7\linewidth]{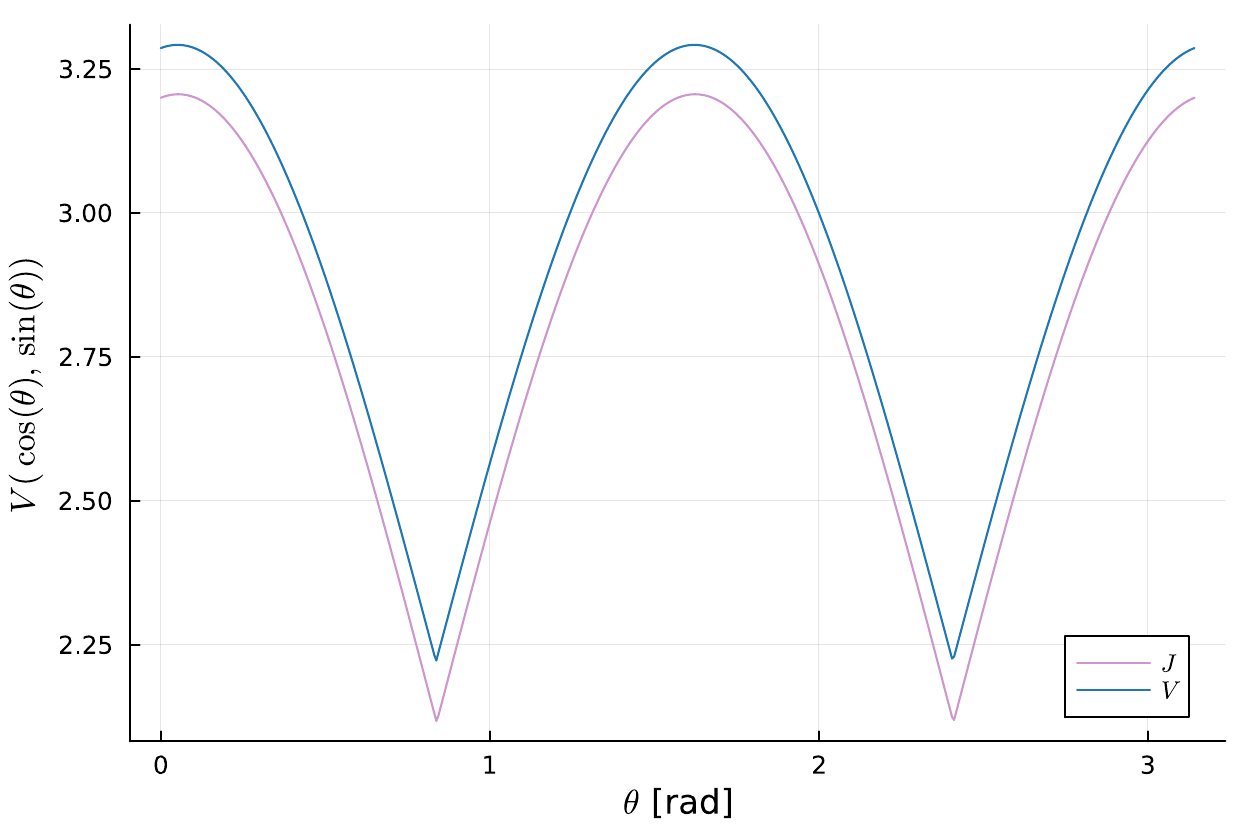}
\caption{Upper bound 
$V(x) = \max_{\alpha\in\{1,2\}} x^\top P_\alpha x$ on the value function of the switched linear system with matrices in~\eqref{eq:matrices_simple_ex}, plotted along the unit circle $x=(\cos(\theta), \sin(\theta))$ for $\theta\in[0,\pi]$.  
The bound is computed using the path-complete framework with the co-complete graph in \cref{fig:pc-graph}, and the true value function $J$ is shown for comparison.}

\label{fig:illustrative_example}
\end{figure}
The LMIs formulation~\eqref{eq:ineq_V_lmi} provides an efficient way to compute upper bounds on the value function of switched linear systems with a quadratic cost. Indeed, solving these LMIs can be done in time polynomial in $n$, $\lvert\calS\rvert$ and $\lvert\calE\rvert$, using interior-point algorithms:

\begin{proposition}[Complexity]
Consider the objective function~\eqref{eq:obj_log_det} with $\ell=\mathrm{trace}$. Then, the corresponding semidefinite program defined by the LMIs~\eqref{eq:ineq_V_lmi} can be solved in
\[
\mathcal O\left(\left(1+|\calE| n\right)^{\frac{1}{2}}n^5 |\calS| \left (n|\calS|^2+n|\calS ||\calE|+|\calE|\right)\right)
\]
basic arithmetic operations.
\end{proposition}

\begin{proof}
Since the objective function is linear, the result follows directly from the complexity analysis in~\cite[Section~6.6.3]{ben2001lectures}, applied to a problem with $|\calE|$ LMIs involving matrices of size $n \times n$ and $n^2|\mathcal{S}|$ decision variables.
\end{proof}

Besides the efficiency of computation, the question of the accuracy (or tightness) of the computed upper bound is central.
This question is addressed in the next two subsections.
\subsection{Convergence on dual De Bruijn graphs} \label{sec:debruijn}

De Bruijn graphs form a hierarchical family of complete graphs, 
originally introduced in~\cite{de1948combinatorial}.
Their definition is recalled below:

\begin{definition}[De Bruijn] \label{def:debruijn}
The \emph{(primal) De Bruijn graph} of order $l\in\Ne$ on $\Mset$, denoted by
$\calH_l(M)=(\calS_l,\calE_l)$, is defined as follows: $\calS_l=\Mset^l$, and $(\alpha,\beta,i)\in\calE_l$ if and only if $\alpha=(j_1,\ldots,j_l)$ and $\beta=(i,j_1,\ldots,j_{l-1})$ for some $(j_1,\ldots,j_l)\in\Mset^l$.
\end{definition}

The De Bruijn graph $\calH_l(M)$ is complete, and its dual, denoted by $\calH_l^\top(M)$, is co-complete. Below, we establish a convergence result for dual De Bruijn graphs:

\begin{theorem}\label{thm:convergence}
Consider system~\eqref{eq:linear_autonomous} and a quadratic cost function $c(x)=x^\top Qx$ with $Q\in\Re^\nxn_{\succ0}$.
Assume that system~\eqref{eq:linear_autonomous} is asymptotically stable.
For each $l\in\Ne$, let $\calH_l^\top(M) = (\calS_l, \calE_l)$ be the dual De Bruijn graph of order~$l$ on $\Mset$ (see~\cref{def:debruijn}).
Then, there exists a sequence of sets of quadratic functions 
\(\{\{V_\alpha^l\}_{\alpha \in \calS_l}\}_{l \in \Ne}\) such that for each $l$, \(\{V_\alpha^l\}_{\alpha\in\calS_l}\subseteq \calT_\txtquad\) satisfies~\eqref{eq:ineq_V} with $\calE=\calE_l$, and
\[
\lim_{l\to\infty} V^l(x) = J(x),
\quad 
V^l(x) \coloneqq \max_{\alpha\in\calS_l} V_\alpha^l(x),
\quad 
\forall\,x\in\Re^n,
\]
where $J$ is the value function of system~\eqref{eq:linear_autonomous} with cost $c$.
\end{theorem}

\begin{remark}
Theorem~\ref{thm:convergence} guarantees the existence of a sequence of sets of quadratic functions 
\(\{\{V_\alpha^l\}_{\alpha \in \calS_l}\}_{l \in \Ne}\), each set satisfying the inequalities~\eqref{eq:ineq_V} on $\calH^\top_l(M)$, 
such that the corresponding sequence of functions 
\(V^l(x) = \max_{\alpha \in \calS_l} V_\alpha^l(x)\) converges to the true value function \(J(x)\).  
While this ensures that a non-conservative feasible solution exists in the limit, the solution obtained by solving the actual optimization problem may be conservative, as it depends on the chosen objective function; see Remark~\ref{rem:objective}.\hfill$\triangleleft$
\end{remark}

The proof of \cref{thm:convergence} uses the following lemma, which is a consequence of the system being asymptotically stable:

\begin{lemma}{{\cite[Proposition~2.13]{sun2011stability}}}\label{lem:stable}
Consider system~\eqref{eq:linear_autonomous}, and assume that it is asymptotically stable. Then, there are $0\leq\rho<1$ and $C\geq0$ such that for all $x\in\Re^n$, $\sigma:\Ne\to\Mset$ and $k\in\Ne$, $\lVert\xi(k,x,\sigma)\rVert\leq C\rho^k\lVert x\rVert$. 
\end{lemma}

\begin{proof}[Proof of Theorem~\ref{thm:convergence}]
For each $l\in\Ne$, define
\[
\eta_l \coloneqq \max_{x\in\Re^n\setminus\{0\}}\max_{\sigma:\Ne\to\Mset}\frac{\xi(l+1,x,\sigma)^\top Q\xi(l+1,x,\sigma)}{x^\top Qx}.
\]
By Lemma~\ref{lem:stable}, it holds that for all $l\in\Ne$, $\eta_l\leq \frac{\lambda_{\mathrm{max}}(Q)}{\lambda_{\mathrm{min}}(Q)}C^2\rho^{2l+2}$. 
This implies that there exists $L\in \Ne$ such that, for all $l \geq L$, $\eta_l < 1$. We will show that:\\\\
\emph{Claim~1:} For all $l \geq L$, there is $\{V^l_\alpha\}_{\alpha\in\calS_l}\subseteq\calT_\txtquad$ satisfying~\eqref{eq:ineq_V} with $\calE=\calE_l$ and such that $V^l(x)\leq\frac1{1-\eta_l} J(x)$ for all $x\in\Re^n$, where $V^l(x)\coloneqq\max_{\alpha\in\calS_l} V_\alpha^l(x)$.\\\\
\emph{Proof of Claim~1: }
Fix $l\geq L$.
For each $\alpha=(i_0,\ldots,i_{l-1})\in\calS_l$, define $\widetilde V_\alpha^l:\Re^n\to\Re_{\geq0}$ by
\begin{equation}\widetilde V_\alpha^l(x)\coloneqq\sum_{k=0}^{l}\xi(k,x,\sigma_\alpha)^\top Q\xi(k,x,\sigma_\alpha),
\label{eq:def_V_tilde}
\end{equation}
where $\sigma_\alpha:\Ne\to\Mset$ satisfies $\sigma_\alpha(k)=i_k$ for all $k\in\{0,\ldots,l-1\}$.
It is clear that for each $\alpha\in\calS_l$, $\widetilde V_\alpha^l\in\calT_\txtquad$.
Also, note that for all $x\in\Re^n$ and $\alpha\in\calS_l$, $\widetilde{V}^l_\alpha(x)\leq J(x)$.
Hence,
\begin{equation}\max_{\alpha\in\calS_l}\widetilde{V}^l_\alpha(x)\leq J(x),\quad \forall \, x \in \Re^n.\label{eq:proof_conv}
\end{equation}
Now, for each $\alpha\in\calS_l$, let us define $V_\alpha^l\coloneqq\frac1{1-\eta_l}\widetilde{V}_\alpha^l$.
From~\eqref{eq:proof_conv}, it follows that
\[
V^l(x)\coloneqq \max_{\alpha \in \calS_l} V_\alpha^l(x) \le \frac{1}{1 - \eta_l} J(x), \quad \forall \, x \in \Re^n.
\]
To conclude the proof of Claim~1, it remains to show that $\{V^l_\alpha\}_{\alpha\in\calS_l}$ satisfies~\eqref{eq:ineq_V} with $\calE=\calE_l$.
Therefore, let us consider an arbitrary edge $(\alpha,\beta,i_0)\in \calE_l$.
Since we consider a dual De Bruijn graph, we have $\alpha=(i_0,i_1,\ldots,i_{l-1})$ and $\beta=(i_1,\ldots, i_l)$ for some $(i_1,\ldots,i_l)\in\Mset^l$.
By~\eqref{eq:def_V_tilde}, it holds that
\begin{align*}
&\widetilde V_\alpha^l(x)=\widetilde V_\beta^l(A_{i_0}x)+x^\top Qx - \xi(l+1,x,\bar\sigma)^\top Q\xi(l+1,x,\bar\sigma), \quad \forall\,x\in\Re^n,
\label{eq:eq1_thm1}
\end{align*}
where $\bar\sigma:\Ne\to\Mset$ satisfies $\bar\sigma(k)=i_k$ for $k\in\{0,\ldots,l\}$.
Hence, from the definition of $\eta_l$, it follows that
\begin{equation}
\widetilde V_\alpha^l(x)\geq\widetilde V_\beta^l(A_{i_0}x)+(1-\eta_l)x^\top Qx,\quad \forall\,x\in\Re^n.\label{eq:proof_conv2}
\end{equation}
Using $V_\alpha^l\coloneqq \frac1{1-\eta_l}\widetilde{V}_\alpha^l$ and $V_\beta^l\coloneqq\frac1{1-\eta_l}\widetilde{V}_\beta^l$, we obtain
\[
V_\alpha^l(x) \ge V_\beta^l(A_{i_0} x) + x^\top Q x, \quad \forall\, x \in \Re^n,
\]
which is~\eqref{eq:ineq_V}.
Since the edge $(\alpha,\beta,i_0)$ was arbitrary, this concludes the proof of the claim.~\hfill$\lrcorner$\\\\
We use Claim~1 to conclude the proof of the theorem.
For each $l\geq L$, let $\{V^l_\alpha\}_{\alpha\in\calS_l}$ and $V^l$ be as in
Claim~1.
The claim states that for each $l\geq L$, $\{V^l_\alpha\}_{\alpha\in\calS_l}$ satisfies~\eqref{eq:ineq_V} with $\calE=\calE_l$ and $(1-\eta_l) V^l(x) \leq J(x) \leq V^l(x)$ (where the second inequality is obtained from \cref{theorem:ub_complete_cocomplete} since $V^l(x)=\max_{\alpha\in\calS_l} V^l_\alpha(x)$ and $\calH_l^\top(M)$ is co-complete).
Since $\lim_{l\to\infty} \eta_l = 0$, it follows that $\lim_{l\to\infty} V^l(x) = J(x)$, which concludes the proof of the theorem.
\end{proof}

\subsection{Relative accuracy of the upper bounds} \label{sec:approximation-guarantees}

In this subsection, we present an algorithmic framework to assess the \emph{relative tightness} of the upper bounds derived from Theorem~\ref{theorem:ub_complete_cocomplete} for switched linear systems with quadratic cost and quadratic template. Specifically, given an upper bound $V$ of the form~\eqref{eq:def_V_bar_complete} or~\eqref{eq:def_V_bar_cocomplete} with $\calT=\calT_\txtquad$, we provide a tractable approach to compute a factor $\mu\geq1$ such that 
\begin{equation}
    \frac{1}{\mu} V(x) \leq J(x) \leq V(x),\quad \forall\, x \in \Re^n.
\end{equation}
To this end, we determine the smallest $\mu$ such that $W \coloneqq \frac{1}{\mu} V$ satisfies the dynamic programming inequality~\eqref{eq:bellman-lower}.
This implies that $W$ constitutes a \emph{lower bound} on the true value function $J$.
Concretely, for a switched system~\eqref{eq:switched_controlled} and a cost $c$, given a previously computed upper bound $V : \Re^n \to \mathbb{R}_{\ge 0}$, the problem reduces to finding the smallest $\mu\geq1$ such that
\begin{equation}\label{eq:accuracy-max}
    V(x) \le \mu\, c(x) + \max_{i \in \Mset} V(f_i(x)), 
    \quad \forall\,x \in \Re^n.
\end{equation}
For both forms of $V$, namely \eqref{eq:def_V_bar_complete} and \eqref{eq:def_V_bar_cocomplete}, we explain below how this can be done efficiently using semidefinite programming when the system is switched linear, the cost quadratic and the template quadratic.

\subsubsection{The max case}

Assume first that $V$ has the form~\eqref{eq:def_V_bar_cocomplete} with quadratic functions, that is, $V(x) = \max_{\alpha \in \calS} V_\alpha(x)$, where $\calS$ is a finite set, and for each $\alpha \in \calS$, $V_\alpha(x) = x^\top P_\alpha x$ with $P_\alpha \in \Re^{n \times n}_{\succeq 0}$.
Consider the switched linear system~\eqref{eq:linear_autonomous} and the quadratic cost $c(x) = x^\top Q x$ with $Q \in \Re^{n \times n}_{\succ 0}$.
To solve~\eqref{eq:accuracy-max} in this setting, we first need to reformulate it in a tractable way.
To this end, we draw inspiration from the \emph{S-procedure}~\cite{boyd2004convex}.\footnote{
The S-procedure allows to rewrite (conservatively) an implication of the form 
$\bigwedge_{i=1}^m x^\top Q_i x \ge 0 \Rightarrow x^\top R x \ge 0$ 
as the existence of nonnegative multipliers $t_1,\ldots,t_m \ge 0$ such that 
$R \succeq \sum_{i=1}^m t_i Q_i$.}
In particular, for each $(\gamma,\alpha,i,\beta,j)\in\calS\times\calS\times\Mset\times\calS\times\Mset$, we introduce a multiplier $t_{\gamma,\alpha,i,\beta,j}\geq0$.
We then consider the condition
\begin{align}
&V_\gamma(x)\leq \mu c(x) + V_\alpha(f_i(x))+\sum_{(\beta,j)\in\calS\times\Mset} t_{\gamma,\alpha,i,\beta,j}\{V_\beta(f_j(x)) - V_\alpha(f_i(x))\}, \nonumber \\
&\hspace{7.5cm}\forall\,x\in\Re^n,\:\forall\,(\gamma,\alpha,i)\in\calS\times\calS\times\Mset. \label{eq:accuracy-ineq-max}
\end{align}
With the hypotheses above on the system (i.e., $f_i(x)=A_ix$ for each $i\in\Mset$), the cost function $c$ and $\{V_\alpha\}_{\alpha\in\calS}$,~\eqref{eq:accuracy-ineq-max} can be formulated as a set of LMIs:
\begin{align}
&P_\gamma \preceq \mu Q + A_i^\top P_\alpha A_i^{} + \sum_{(\beta,j)\in\calS\times\Mset} t_{\gamma,\alpha,i,\beta,j}\big(A_j^\top P_\beta A_j^{} - A_i^\top P_\alpha A_i^{}\big),
\nonumber \\
&\hspace{7.5cm}\forall\,x\in\Re^n,\:\forall\,(\gamma,\alpha,i)\in\calS\times\calS\times\Mset.
\label{eq:lb-cocomplete}
\end{align}
Hence, it can be solved efficiently using interior-point methods. The following result states that any feasible solution to the simplified constraints~\eqref{eq:accuracy-ineq-max} provides a feasible solution to the original constraints~\eqref{eq:accuracy-max}:

\begin{proposition}\label{prop:lower-cocomplete}
Given system~\eqref{eq:switched_controlled}, $\{V_\alpha\}_{\alpha\in\calS}\subseteq \calT$ and a cost function $c:\Re^n\to\Re_{\geq0}$, let $\mu\geq1$ and
\[
\{t_{\gamma,\alpha,i,\beta,j}\}_{(\gamma,\alpha,i,\beta,j)\in\calS\times\calS\times\Mset\times\calS\times\Mset}\subseteq\Re_{\geq0}
\]
satisfy~\eqref{eq:accuracy-ineq-max}.
Then, $\mu$ satisfies~\eqref{eq:accuracy-max} with $V(x)=\max_{\alpha\in\calS}V_\alpha(x)$.
\end{proposition}

\begin{proof}
Consider an arbitrary $x\in\Re^n$.
Let $\gamma\in\calS$ such that
\[
V_\gamma(x)=\max_{\zeta\in\calS}V_\zeta(x),\]
and let $(\alpha,i)\in\calS\times\Mset$ such that
\[
V_\alpha(f_i(x))=\max_{j\in\Mset}\max_{\beta\in\calS}V_\beta(f_j(x)).
\]
Then, by~\eqref{eq:accuracy-ineq-max}, it follows that $V_\gamma(x)\leq\mu c(x)+V_\alpha(f_i(x))$, which implies~\eqref{eq:accuracy-max} by the definitions of $\gamma$, $\alpha$ and $i$.
Since $x$ was arbitrary, this concludes the proof.
\end{proof}

\subsubsection{The min case}

Now, assume that $V$ has the form~\eqref{eq:def_V_bar_complete} with quadratic functions, that is, $V(x) = \min_{\alpha \in \calS} V_\alpha(x)$, where $\calS$ is a finite set and, for each $\alpha \in \calS$, $V_\alpha(x) = x^\top P_\alpha x$ with $P_\alpha \in \Re^{n \times n}_{\succeq 0}$. Consider the switched linear system~\eqref{eq:linear_autonomous} and the quadratic cost $c(x) = x^\top Q x$ with $Q \in \Re^{n \times n}_{\succ 0}$.
Solving~\eqref{eq:accuracy-max} in this setting can be treated in a way similar to the max case, but more multipliers are needed. In particular, for each $(\gamma,\alpha,i,\vec\omega,\zeta)\in\calS\times\calS\times\Mset\times\calS^M\times\calS$, let $s_{\gamma,\alpha,i,\vec\omega,\zeta}\geq0$, and for each $(\gamma,\alpha,i,\vec\omega,j)\in\calS\times\calS\times\Mset\times\calS^M\times\Mset$, let $t_{\gamma,\alpha,i,\vec\omega,j}\geq0$.
Consider the condition:
\begin{align}
&V_\gamma(x) + \sum_{\zeta\in\calS}s_{\gamma,\alpha,i,\vec\omega,\zeta}\{V_\zeta(x)-V_\gamma(x)\} \leq \mu c(x) + V_\alpha(f_i(x))+\sum_{j\in\Mset} t_{\gamma,\alpha,i,\vec\omega,j}\{V_{\vec\omega_j}(f_j(x)) - V_\alpha(f_i(x))\}, \nonumber \\
&\hspace{7.5cm}\forall\,x\in\Re^n,\:\forall\,(\gamma,\alpha,i,\vec\omega)\in\calS\times\calS\times\Mset\times\calS^M.
\label{eq:accuracy-ineq-min}
\end{align}
With the hypotheses above on the system (i.e., $f_i(x)=A_ix$ for each $i\in\Mset$), the cost function $c$ and $\{V_\alpha\}_{\alpha\in\calS}$,~\eqref{eq:accuracy-ineq-min} can be formulated as a set of LMIs:
\begin{align}
&P_\gamma + \sum_{\zeta\in\calS} s_{\gamma,\alpha,i,\vec\omega,\zeta}\{P_\zeta-P_\gamma\} \preceq \mu Q + A_i^\top P_\alpha A_i+\sum_{j\in\Mset} t_{\gamma,\alpha,i,\vec\omega,j}\{A_j^\top P_{\vec\omega_j} A_j^{} - A_i^\top P_\alpha A_i^{}\}, \nonumber \\
&\hspace{7.3cm}\forall\,x\in\Re^n,\:\forall\,(\gamma,\alpha,i,\vec\omega)\in\calS\times\calS\times\Mset\times\calS^M.
\end{align}
Hence, it can be solved efficiently using interior-point methods.
The following result states that any feasible solution to~\eqref{eq:accuracy-ineq-min} provides a feasible solution to~\eqref{eq:accuracy-max}:

\begin{proposition}
Given system~\eqref{eq:switched_controlled}, $\{V_\alpha\}_{\alpha\in\calS}\subseteq \calT$ and a cost function $c:\Re^n\to\Re_{\geq0}$, let $\mu\geq1$,
\begin{align*}
&\{s_{\gamma,\alpha,i,\vec\omega,\zeta}\}_{(\gamma,\alpha,i,\vec\omega,\zeta)\in\calS\times\calS\times\Mset\times\calS^M\times\calS}\subseteq\Re_{\geq0} \quad \text{and} \quad \{t_{\gamma,\alpha,i,\vec\omega,j}\}_{(\gamma,\alpha,i,\vec\omega,j)\in\calS\times\calS\times\Mset\times\calS^M\times\Mset}\subseteq\Re_{\geq0}
\end{align*}
satisfy~\eqref{eq:accuracy-ineq-min}.
Then, $\mu$ satisfies~\eqref{eq:accuracy-max} with $V(x)=\min_{\alpha\in\calS}V_\alpha(x)$.
\end{proposition}

\begin{proof}
Consider an arbitrary $x\in\Re^n$.
Let $\gamma\in\calS$ such that
\[
V_\gamma(x)=\min_{\zeta\in\calS}V_\zeta(x),
\]
and let $(\alpha,i)\in\calS\times\Mset$ such that
\[
V_\alpha(f_i(x))=\max_{j\in\Mset}\min_{\beta\in\calS}V_\beta(f_j(x)).
\]
This implies that for each $j\in\Mset$, there exists $\beta_j\in\calS$ such that $V_{\beta_j}(f_j(x))\leq V_\alpha(f_i(x))$.
Let $\vec\omega=(\beta_1,\ldots,\beta_M)$ where the $\beta_j$'s are as above.
Then, by~\eqref{eq:accuracy-ineq-min}, we have that $V_\gamma(x)\leq\mu c(x)+V_\alpha(f_i(x))$, which implies~\eqref{eq:accuracy-max} by the definitions of $\gamma$, $\alpha$ and $i$.
Since $x$ was arbitrary, this concludes the proof.
\end{proof}

Thanks to the preceding propositions, the scaling factor~$\mu$, which bounds the relative accuracy of the upper bound~$V$, can be computed efficiently through LMIs.

\section{Control of arbitrarily switched linear systems}\label{sec:path-complete-control}

In this section, we use the path-complete framework for the optimal control of switched linear systems under arbitrary switching, with control affine input (i.e., in the form $B_iu$).
We will show that the path-complete framework can be used to provide upper bounds on the optimal value function (i.e., the worst-case cost-to-go under optimal feedback control) for these systems.\\\\
Concretely, we consider controlled switched linear systems, i.e., systems of the form
\begin{equation}
x_{k+1}=A_{\sigma(k)}x_k+B_{\sigma(k)}u_k, \quad k\in\Ne,
\label{eq:switched_Ax_Bu}
\end{equation}
where $u_k\in\Re^m$ is the \emph{control input} at time $k$ and for all $i\in\Mset$, $A_i\in\Re^\nxn$ and $B_i\in\Re^\nxm$.\\\\
We consider a quadratic cost function $\ct:\Re^n\times\Re^m\to\Re_{\geq0}$ defined as
\begin{equation}
\ct(x,u)=x^\top Qx+u^\top Ru, \quad Q\in\Re^\nxn_{\succ0},\:R\in\Re^\mxm_{\succ0}.
\label{eq:cost_x_u}
\end{equation}
Given a feedback policy $\phi:\Re^n\to\Re^m$, the associated closed-loop system is the switched system~\eqref{eq:switched_Ax_Bu} with $u_k=\phi(x_k)$ for all $k\in\Ne$.
Equivalently, it is the autonomous switched system~\eqref{eq:switched_controlled} where for each $i\in\Mset$, $f_i$ is given by $f_i(x)=A_ix+B_i\phi(x)$.
The associated closed-loop cost is defined by $c^\phi(x)=\ct(x,\phi(x))$.
The \emph{value function} of the closed-loop system is defined as follows.
\begin{definition}[Closed-loop value function]
Given a feedback policy $\phi:\Re^n\to\Re^m$, the \emph{closed-loop value function} of system~\eqref{eq:switched_Ax_Bu} with cost $\ct:\Re^n\times\Re^m\to\Re_{\geq0}$ and policy $\phi$ is denoted by $J^\phi(x)$ and defined as the value function (Definition~\ref{definition:cost_to_go_val_function}) of system~\eqref{eq:switched_controlled} with $f_i$ defined by $f_i(x)\coloneqq A_ix+B_i\phi(x)$ for all $x\in\Re^n$ and $i\in\Mset$, with cost $c(x)\coloneqq\ct(x,\phi(x))$ for all $x\in\Re^n$.
The \emph{optimal value function} of system~\eqref{eq:switched_Ax_Bu} with cost $\ct$ is the smallest value function that can be obtained (pointwise) among all closed-loop systems:
\[
J^*(x)\coloneqq\inf_{\phi:\Re^n\to\Re^m} J^\phi(x).
\]\vskip0pt \label{def:closed_loop_val_fct}
\end{definition}

In this work, we do not attempt to compute the optimal feedback policy. 
Instead, we consider a \emph{piecewise linear} policy because it allows us to compute an upper bound on the value function of the associated closed-loop system by solving LMIs, 
similar to what we did for the autonomous case. 
By construction, this bound also provides an upper bound on the optimal value function $J^*$.
In particular, we adopt the policy proposed in~\cite{dellarossa2024graph}, 
where $|\calS|$ gain matrices 
$\{K_\alpha\}_{\alpha\in\calS}\subseteq\Re^\mxn$ are introduced, one for each node of a complete graph 
$\calG=(\calS,\calE)$.
The feedback policy is defined as
\begin{equation}
\phi(x) = K_{\kappa(x)} x, \quad \kappa(x) \in \argmin_{\alpha\in\calS} V_\alpha(x),
\label{eq:policy}
\end{equation}
where $\{V_\alpha\}_{\alpha\in\calS}\subseteq\calF^n_{\geq0}$ are functions associated with the nodes of the graph.
The resulting closed-loop dynamics is given by
\begin{equation}
x_{k+1} = \left(A_{\sigma(k)} + B_{\sigma(k)} K_{\kappa(x_k)}\right) x_k,
\label{eq:LQR_switched_system}
\end{equation}
and the resulting cost $c:\Re^n\to\Re_{\geq0}$ by
\begin{equation}
c(x)=x^\top Qx+x^\top K_{\kappa(x)}^\top R K_{\kappa(x)}^{}x, \quad Q\in\Re^\nxn_{\succ0},\:R\in\Re^\mxm_{\succ0}.
\label{eq:cost_controlled}
\end{equation}

The following theorem shows that if $\calG$ is complete, and $\{V_\alpha\}_{\alpha\in\calS}$ satisfy graph-driven inequalities similar to~\eqref{eq:ineq_V}, then an upper bound analogous to that in the autonomous case can be obtained by taking the pointwise minimum of the $V_\alpha$'s:

\begin{theorem}\label{theorem:upper-bound-control}
Consider system~\eqref{eq:switched_Ax_Bu}, a cost function $\ct:\Re^n\times\Re^m\allowbreak\to\Re_{\geq0}$, a directed labeled graph $\calG=(\calS,\calE)$ and a template $\calT\subseteq\calF^n_{\geq0}$.
Assume that $\calG$ is complete for $\Mset$ and that $\{V_\alpha\}_{\alpha\in\calS}\subseteq\allowbreak\calT$, $\{K_\alpha\}_{\alpha\in\calS}\subseteq \Re^{m\times n}$ satisfy
\begin{align}
&V_\alpha(x) \ge \ct(x,K_\alpha x) + V_\beta(A_i x+ B_i K_\alpha x), \quad \forall\,x\in\Re^n,\:\forall\, (\alpha,\beta,i)\in\calE.
\label{eq:PC_upper_bound_LQR}
\end{align}
Let $V:\Re^n\to\Re_{\geq0}$ be defined by
\begin{equation}\label{eq:min-control-value}
V(x) \coloneqq \min_{\alpha\in\calS} V_\alpha(x),
\end{equation}
and let $\phi:\Re^n\to\Re^m$ be defined as in~\eqref{eq:policy}.
Then, for all $x\in \Re^n$, $V(x)\geq J^\phi(x)\geq J^*(x)$, where $J^*$ is the optimal value function of system~\eqref{eq:switched_Ax_Bu} with cost $\ct$, and $J^\phi$ is the value function of the associated closed-loop system with policy $\phi$.
\end{theorem}

\begin{proof}
From~\eqref{eq:PC_upper_bound_LQR} and the definition of $V$, we have
\begin{equation}
V_\alpha(x) \ge \ct(x,K_\alpha x) + V(A_i x + B_i K_\alpha x), 
\quad \forall\,x\in\Re^n,\ \forall\,(\alpha,i)\in E,
\end{equation}
where the set $E$ is defined as in the \emph{complete} part of the proof of~\cref{theorem:ub_complete_cocomplete}:
\[
E = \{(\alpha,i)\in\calS\times\Mset:\ \exists\,\beta \in\calS \text{ such that } (\alpha, \beta,i)\in \calE\}.
\]
The graph being complete, we rewrite the previous set of inequalities as
\begin{align*}
&V_\alpha(x) \geq \ct(x,K_\alpha x) + V(A_i x+ B_i K_\alpha x), \quad \forall\,x\in\Re^n,\:\forall\,\alpha\in\calS,\:\forall\,i\in\Mset.
\end{align*}
Taking $\alpha=\kappa(x)$ with $\kappa$ as in~\eqref{eq:policy} and letting $c(x)\coloneqq\ct(x,K_{\kappa(x)}x)$, we obtain
\begin{equation}
V_{\kappa(x)}(x) \ge c(x) + V(A_i x+ B_i K_{\kappa(x)}x), \quad \forall\,x\in\Re^n, \:\forall\,i\in\Mset. 
\end{equation} 
By definition of $V$ and $\kappa(x)$, we have
\begin{equation}
V (x) \ge c(x) + V(A_i x+ B_i K_{\kappa(x)}x), \quad \forall\,x\in\Re^n,\:\forall\,i\in\Mset. 
\end{equation} 
Taking the maximum over $i$, we get
\begin{equation}
V(x)\geq c(x) + \max_{i\in\Mset} V(A_i x+ B_i K_{\kappa(x)} x), \quad \forall\,x\in\Re^n.
\end{equation}
Defining $f_i(x) \coloneqq A_i x + B_i K_{\kappa(x)}x$, we use~\cref{prop:bellman-upper} to conclude that $V(x)\ge J^\phi(x)$ for all $x\in \Re^n$. The inequality $J^\phi(x)\ge J^*(x)$ for each $x\in \Re^n$ follows directly from~\cref{def:closed_loop_val_fct}, completing the proof.
\end{proof}
Now, we show that the inequalities~\eqref{eq:PC_upper_bound_LQR} can be expressed as LMIs.
As for autonomous linear switched systems (see~\cref{section:LMIs_formulation_linear_auton_syst}), we assume a quadratic template, i.e., $V_\alpha\in\calT_\txtquad$ for each $\alpha\in\calS$.
Again, we parametrize them as $V_\alpha(x) = x^\top P_\alpha x$, with $P_\alpha \in \Re^{n \times n}_{\succeq 0}$ for all $\alpha \in \mathcal{S}$. Using this parametrization together with the expression of the cost in~\eqref{eq:cost_x_u}, 
we can rewrite~\eqref{eq:PC_upper_bound_LQR} as follows:
\begin{align}
&P_\alpha - Q - K_\alpha^\top R K_\alpha - A_i^\top P_\beta A_i - K_\alpha^\top B_i^\top P_\beta B_i K_\alpha  - 2 K_\alpha^\top B_i^\top P_\beta A_i \succeq 0, \quad \forall\, (\alpha,\beta,i)\in\calE.
\label{eq:matrix_ineq_nonlinear}
\end{align}

The following proposition states that these matrix inequalities can be expressed as LMIs through appropriate changes of variables, and draws inspiration from~\cite[Theorem 1]{kothare1996robust} and~\cite[Lemma 11]{dellarossa2024graph}.

\begin{proposition}\label{prop:schur-for-control}
Given $M\in \Ne_{>0}$ and a finite set $\calS$, sets of matrices $\{A_i\}_{i\in \Mset}\subseteq \Re^{\nxn}$, $\{B_i\}_{i\in \Mset}\subseteq \Re^{\nxm}$, $\{K_\alpha\}_{\alpha\in\calS}\subseteq \Re^{m\times n}$ and $\{P_\alpha\}_{\alpha\in\calS}\subseteq \Re^\nxn_{\succ0}$, and matrices $Q\in\Re^\nxn_{\succ0}$ and $R\in\Re^\mxm_{\succ0}$, consider the transformations
\[
S_\alpha^{} = P_\alpha^{-1}, 
\quad 
Y_\alpha = K_\alpha P_\alpha^{-1}, \quad \forall\,\alpha\in\calS.
\]
The inequalities~\eqref{eq:matrix_ineq_nonlinear} are equivalent to
\begin{align}
&\begin{bmatrix}
S_\alpha & S_\alpha A_i^\top + Y_\alpha^\top B_i^\top & S_\alpha & Y_\alpha^\top \\
A_i S_\alpha + B_i Y_\alpha & S_\beta & 0 & 0 \\
S_\alpha & 0 & Q^{-1} &0 \\
Y_\alpha & 0 & 0 & R^{-1}
\end{bmatrix} \succeq0, \quad \forall\,(\alpha,\beta,i)\in\calE. \label{eq:lmi_controlled}
\end{align}\vskip0pt
\end{proposition}

\begin{proof}
The matrix in~\eqref{eq:lmi_controlled} can be written in block form as
\[
G =
\begin{bmatrix}
A & B \\
B^\top & C
\end{bmatrix},
\]
where
\[
A = S_\alpha, \;\;
B = \left[S_\alpha A_i^\top + Y_\alpha^\top B_i^\top \;\; S_\alpha \;\; Y_\alpha^\top\right], \;\;
C = \begin{bmatrix}
S_\beta & 0 & 0 \\
0 & Q^{-1} & 0 \\
0 & 0 & R^{-1}
\end{bmatrix}.
\]
Since $C\succ 0$, by the Schur complement, the condition $G \succeq 0$ 
is equivalent to
\[
A-BC^{-1}B^\top\succeq 0.
\]
Substituting $P_\alpha^{} = S_\alpha^{-1}$ and $K_\alpha = Y_\alpha P_\alpha$, one recovers inequality~\eqref{eq:matrix_ineq_nonlinear}. 
This establishes the equivalence between \eqref{eq:matrix_ineq_nonlinear} and \eqref{eq:lmi_controlled}.
\end{proof}

\section{Numerical experiments on switched linear systems}\label{sec:experiments}

In this section, we illustrate the proposed approach through numerical experiments, in both autonomous and controlled cases.

\subsection{Autonomous case}

In this subsection, we first consider an autonomous switched linear system of the form~\eqref{eq:linear_autonomous} with two states and two switching modes, whose matrices are given in~\eqref{eq:matrices_simple_ex}, and we set $Q = I$ in the cost function $c(x)=x^TQx$. The procedure is as follows.\\\\
For graph orders $l = 1, 2, 3$:
\begin{enumerate}
    \item We consider the co-complete dual De Bruijn graph $\mathcal{H}^\top_l(2) = (\calS_l, \mathcal{E}_l)$ (see~\cref{def:debruijn}).
    \item Using the inequalities in~\eqref{eq:ineq_V_lmi} and the objective function~\eqref{eq:obj_log_det} with $\ell=\mathrm{trace}$, we compute $\{P_\alpha\}_{\alpha\in\calS_l}$ by solving the semidefinite program
    \begin{equation}
    \min_{\{P_\alpha\}_{\alpha\in\calS_l}} \sum_{\alpha\in\calS_l} \mathrm{trace}(P_\alpha) \quad \text{s.t.} \quad \text{\eqref{eq:ineq_V_lmi} holds with $\calE=\calE_l$}.
    \end{equation}
    \item By~\cref{theorem:ub_complete_cocomplete}, we obtain the upper bound 
    \begin{equation} 
        V^l(x) \coloneqq \max_{\alpha\in\calS_l} x^\top P_\alpha x.
    \end{equation}
    \item We then compute the lower bound $W^l \coloneqq \frac{1}{\mu} V^l$ using~\cref{prop:lower-cocomplete}, by solving
    \begin{equation}
    \min_{\mu\geq1,\,\{t_{\gamma,\alpha,i,\beta,j}\}\subseteq \Re_{\ge0}} \mu \quad
    \text{s.t.} \quad \text{\eqref{eq:lb-cocomplete} holds with $\calS=\calS_l$.}
    \end{equation}
\end{enumerate}
The resulting upper and lower bounds are displayed in~\cref{fig:bounds_debruijn_auton} for $x=(\cos(\theta),\sin(\theta))$, with $\theta\in[0,\pi]$. 
Note that we focus on dual De Bruijn graphs in this autonomous case, since the convergence result was established specifically for this graph family (see~\cref{thm:convergence}). 
As we can see in the figure, when the graph order increases, the gap between the upper and lower bound curves decreases, demonstrating the efficiency of the proposed upper bounding method and its ability to tightly approximate the true value function.\\\\
\begin{figure}
\centering
\includegraphics[width=0.8\linewidth]{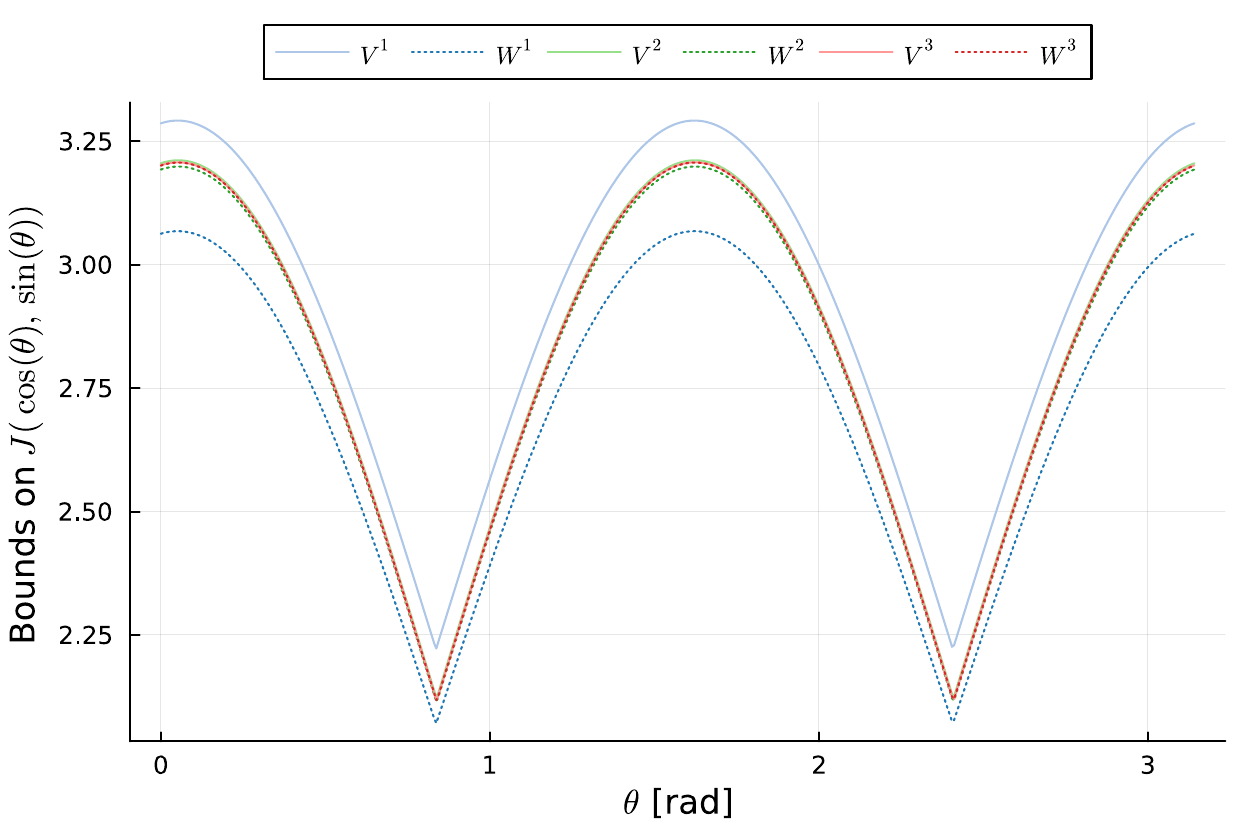} \caption{Upper and lower bounds on the value function, $V^{l}(x)$ and $W^{l}(x)\coloneqq \frac{1}{\mu} V^l(x)$, 
for $x=(\cos(\theta),\,\sin(\theta))$ with $\theta\!\in\![0,\pi]$, 
for the autonomous switched linear system with matrices given in~\eqref{eq:matrices_simple_ex}, and using dual De Bruijn graphs of order $l\!\in\!\{1,2,3\}$. The gap between the upper and lower bound curves decreases as the order increases.}
\label{fig:bounds_debruijn_auton}
\end{figure}To further quantify the tightness of our upper bounds, we analyze the scaling coefficient $\mu$ for various combinations of the number of states, switching modes, and graph orders. Recall that
\begin{equation}
    \frac{1}{\mu} V(x) \leq J(x) \leq V(x),\quad \forall\, x \in \Re^n, \label{eq:tightness_ineq}
\end{equation}
so that values of $\mu$ close to $1$ indicate that the computed upper bound is tight. 
To do that, we use the same procedure as described above, but now the system matrices $A_i$ are generated randomly. For each configuration $(n,M)$, we perform $500$ realizations and report the average results in~\cref{tab:benchmark_autonomous}. Each realization is evaluated for graph order $l=1,2,3,4$. The results show that $\mu$ consistently converges to $1$ as $l$ increases, which illustrates the non-conservativeness of our upper bounds.

\begin{table}[ht]
\centering
\begin{tabular}{@{}c|ccc@{}}
\toprule
\textbf{Graph order} & \textbf{($n=2$, $M=2$)} & \textbf{($n=5$, $M=3$)} & \textbf{($n=8$, $M=2$)}  \\
\midrule
$1$ & $1.038$ & $1.131$& $1.301$\\
$2$ & $1.008$ & $1.032$ &  $1.098$\\
$3$ & $1.002$ & $1.009$ &  $1.031$\\
$4$ & $1.0007$ & $1.002$ &  $1.009$\\
\bottomrule
\end{tabular}\vskip2pt
\caption{Values of $\mu$ in~\eqref{eq:tightness_ineq} for an autonomous switched linear system with various numbers of states ($n$) and switching modes ($M$), computed for several orders of the dual De Bruijn graph. The system matrices are randomly generated, and the reported values are averaged over $500$ realizations. We observe that $\mu$ tends to $1$ as the order increases.}
\label{tab:benchmark_autonomous}
\vskip-12pt
\end{table}

\subsection{Controlled case}

In this subsection, we illustrate our approach on a controlled switched linear system of the form~\eqref{eq:switched_Ax_Bu}. Specifically, we consider the system from~\cite[Example~3]{dellarossa2024graph}, which involves the following matrices:
\begin{equation} A_1 = \begin{bmatrix} 0 & 1 \\ -1 & 0 \end{bmatrix},\quad A_2 = \begin{bmatrix} -1 & 0 \\ 0 & -0.95 \end{bmatrix}, \quad B_1 = B_2 = \begin{bmatrix} 1 \\ 0 \end{bmatrix}. \label{eq:ex_2D_controlled} \end{equation}
Our goal is to find a policy of the form~\eqref{eq:policy} and bound the value function of the resulting closed-loop system, where the cost function is given by $Q = I$ and $R = I$ in~\eqref{eq:cost_controlled}.
For that, we apply the method described in \cref{sec:path-complete-control}. For a graph order $l \in \{ 1, \ldots, 5\}$, we consider the complete De Bruijn graph $\mathcal{H}_l(2) = (\calS_l, \mathcal{E}_l)$ (see~\cref{def:debruijn}), and find matrices $\{P_{\alpha}\}_{\alpha \in \calS_l}$ and $\{K_{\alpha}\}_{\alpha \in \calS_l}$ such that the LMIs \eqref{eq:matrix_ineq_nonlinear} hold. We consider two different objective functions: $\ell=\log \circ \det$ in~\eqref{eq:obj_log_det} (hereafter referred to as the ``log-det'' objective), and the pointwise objective defined in \eqref{eq:pointwise-complete}.\footnote{The first program can be convexified with the Schur complement trick described in \cref{prop:schur-for-control}, and the convexification for the pointwise objective is thoroughly explained in Appendix~\ref{app:dcdc}. The question of whether the objective $\ell=\mathrm{trace}$ in~\eqref{eq:obj_log_det} can be convexified in the controlled case remains open, hence the use of $\ell = \log \circ \det $ in this case.} Solving these programs yields controllers of the form~\eqref{eq:policy} with $\calS=\calS_l$ as well as an upper bound on the value function of the form~\eqref{eq:min-control-value} with $\calS=\calS_l$.\\\\
In \cref{tab:control_experiments}, we report values of $V(x_0)$ obtained using the path-complete method with $l = 1, \dots, 5$, for two initial states, and considering both the ``log-det'' and the pointwise objectives.\footnote{Contrarily to the autonomous case, it is still an open question whether it is possible to derive (useful) approximation guarantees for the controlled case. We therefore directly report the upper bound values $V(x_0)$ in this case.} As one can see, for both initial states, and for both objectives, the upper bound on the value function is decreasing with the order. However, we must precise that, although monotonicity can be observed for both objectives, only monotonicity for the pointwise objective can be proven (see Appendix~\ref{app:monotonic_pointwise} for a proof). We also observe that, as expected, the pointwise bound performs better as we compare the value function at the initial state. Finally, we note that increasing the order of the De Bruijn graph under the ``log-det'' objective is more beneficial (i.e., leads to a greater reduction in the upper bound) for the second initial state $x_0 = (\cos(2), \sin(2))$ than for the first one, $x_0 = (\cos(0.5), \sin(0.5))$.\\\\
\begin{table}
\centering
\begin{tabular}{ccc|cc}
\toprule
 & \multicolumn{2}{c|}{$x_0=(\cos(0.5),\sin(0.5))$} & \multicolumn{2}{c}{$x_0=(\cos(2),\sin(2))$} \\
\cmidrule(r){2-3} \cmidrule(l){4-5}
\textbf{Graph order} & $\textbf{log-det}$ & \textbf{pointwise} & \textbf{log-det} & \textbf{pointwise} \\
\midrule
1 & $10.453$ & $10.200$ & $10.392$ & $9.428$ \\
2 & $10.426$ & $9.475$ & $10.227$ &  $9.004$\\
3 & $10.417$ & $9.218$ & $10.142$ &  $8.841$ \\
4 & $10.412$ & $9.167$ & $10.074$ & $8.814$\\
5 & $10.410$ & $9.163$ & $10.023$ & $8.813$ \\
\bottomrule
\end{tabular}\vskip2pt
\caption{Upper bounds $V^l(x_0)$ on the value function, for two different initial states $x_0$, using De Bruijn graphs of orders $l = 1, \dots, 5$. The system matrices are defined in~\eqref{eq:ex_2D_controlled}. Two objective functions are considered: $\ell=\log \circ \det$ in~\eqref{eq:obj_log_det}, and the pointwise objective function (see~\eqref{eq:pointwise-complete}, complete case). We observe that the upper bounds decrease with the order.}
\label{tab:control_experiments}
\vskip-12pt
\end{table}\\\\
We also display the upper bound obtained under the ``log-det'' objective for $x=(\cos(\theta),\sin(\theta))$, with $\theta\in[0,\pi]$, in \cref{fig:bounds_debruijn_control}. Again, a monotonic decrease {with the order can be observed, although this is not guaranteed for the ``log-det'' objective. Finally, we note that the reduction in the upper bound with increasing order is more pronounced for some points than for others, which is consistent with our previous observations from~\cref{tab:control_experiments}.
\begin{figure}
\centering
\includegraphics[width=0.8\linewidth]{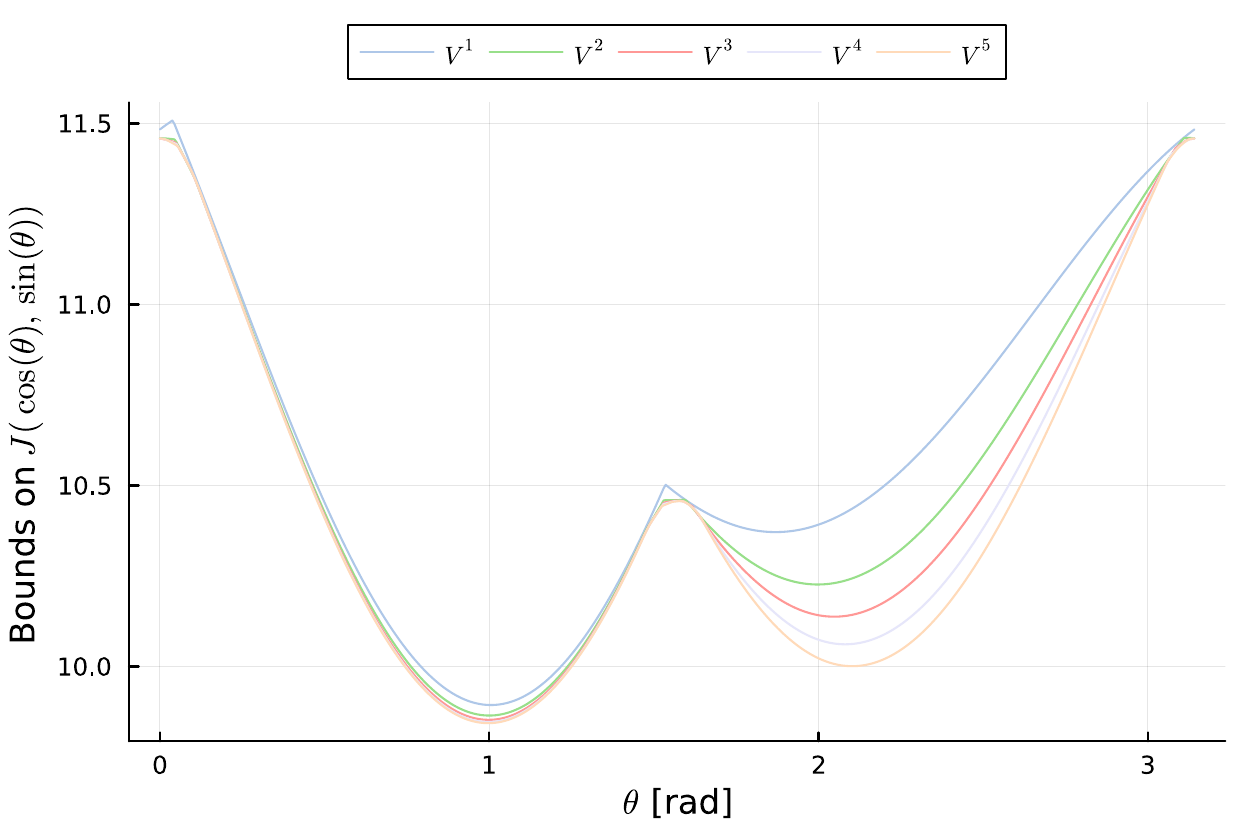} 
\caption{Upper bound $V^{l}(x)$ on the value function, for 
$x = (\cos(\theta), \sin(\theta))$ with $\theta \in [0,\pi]$, 
computed for the controlled switched linear system defined by the matrices in~\eqref{eq:ex_2D_controlled},
using De Bruijn graphs of order $l \in \{1,2,\ldots,5\}$. The upper bound decreases as the order increases, with a more pronounced effect for some points.}
\label{fig:bounds_debruijn_control}
\end{figure}

\section{Conclusion}

We proposed a path-complete framework to compute upper bounds on the value function of switched systems under arbitrary switching. Our approach consists in encoding dynamic programming inequalities on a path-complete graph, associating each node $\alpha$ with a function $V_\alpha$, and combining these functions appropriately (via max or min) to obtain a guaranteed upper bound on the true value function. We first focused on the general nonlinear autonomous case, then specialized to switched linear systems with quadratic costs, for which we derived tractable LMI formulations and provided complexity guarantees. We further quantified the tightness of the bounds and showed that using dual De Bruijn graphs of increasing order yields non-conservative approximations. Finally, we extended the framework to controlled switched linear systems with affine control inputs and demonstrated its efficiency through numerical examples.\\\\
In future work, we plan to compute an approximation guarantee for the controlled setting. This would be very useful since, unlike the autonomous case, it is not trivial to obtain lower bounds by simulating over long horizons (at an exponentially growing computational cost). We also aim to investigate alternative control laws, not necessarily piecewise linear, which could lead to improved policies. Finally, we plan to examine how the upper and lower bounds derived in this paper can be used inside a policy improvement loop.

\bibliographystyle{abbrv}
\bibliography{refs}
\appendix
\section{Re-formulation for SDP solvers} \label{app:dcdc}

In this appendix, we explain how we re-formulate the program 
\begin{equation} \label{eq:program-to-be-convexified}
\min_{\{(P_\alpha, K_{\alpha})\}_{\alpha\in\calS_l}} \min_{\gamma\in\calS_l} x_0^\top P_{\gamma} x_0^{} \quad
    \text{s.t.} \quad 
    \text{\eqref{eq:matrix_ineq_nonlinear} holds with $\calE=\calE_l$}
\end{equation}
to be solved by SDP solvers\color{black}. First, by~\cref{prop:schur-for-control}, we perform the change of variables $S_\alpha^{} \coloneqq P_\alpha^{-1}$ and $Y_\alpha^{} \coloneqq K_\alpha P_\alpha^{-1}$, and the problem becomes
\begin{equation} 
\min_{\{(S_\alpha, Y_{\alpha})\}_{\alpha\in\calS_l}} \min_{\gamma\in\calS_l} x_0^\top S_{\gamma}^{-1} x_0^{} \quad \text{s.t.} \quad \text{\eqref{eq:lmi_controlled} holds with $\calE=\calE_l$}.
\end{equation}
The program above is equivalent to 
\begin{equation}
\min_{\gamma\in\calS_l} \Big[ \underbrace{\min_{\{(S_\alpha, Y_{\alpha})\}_{\alpha\in\calS_l}} x_0^\top S_{\gamma}^{-1} x_0^{} \quad \text{s.t.} \quad \text{\eqref{eq:lmi_controlled} holds with $\calE=\calE_l$}}_{\text{prog}(\gamma)} \Big].
\end{equation}
Therefore one can solve the original problem \eqref{eq:program-to-be-convexified} by solving $\text{prog}(\gamma)$ for all $\gamma$, and select the case where $\gamma$ leads to the best objective value. It remains to re-write the objective value for a given $\gamma$. For that, we introduce a scalar variable $t\in\Re$ as follows: 
\begin{equation} 
\begin{aligned}
    \min_{t,\,\{(S_\alpha, Y_{\alpha})\}_{\alpha\in\calS_l}}\!\! &\quad t \\
    \text{ s.t.} &\quad \text{\eqref{eq:lmi_controlled} holds with $\calE=\calE_l$}, \\
    &\quad t \geq x_0^\top S_{\gamma}^{-1} x_0^{}.
\end{aligned}
\end{equation}
Finally, we can re-write the last constraint using the Schur complement as 
\begin{equation}
    t \geq x_0^\top S_{\gamma}^{-1} x_0^{} \;\Leftrightarrow\; \begin{bmatrix}
        S_\gamma & x_0^{}\\
        x_0^\top & t
    \end{bmatrix} \succeq 0, 
\end{equation}
and solving $\text{prog}(\gamma)$ is equivalent to solving 
\begin{equation} 
\begin{aligned}
    \min_{t,\,\{(S_\alpha, Y_{\alpha})\}_{\alpha\in\calS_l}}\!\! &\quad t \\
    \text{ s.t.} &\quad \text{\eqref{eq:lmi_controlled} holds with $\calE=\calE_l$}, \\
    &\quad \begin{bmatrix}
        S_\gamma & x_0^{}\\
        x_0^\top & t
    \end{bmatrix} \succeq 0. 
\end{aligned}
\end{equation}
The latter is an SDP program that can be handled by most SDP solvers.

\section{Monotonic decrease of the upper bound with respect to the graph order}\label{app:monotonic_pointwise}

In this appendix, we show that the objective function of the following optimization problem:
\begin{equation}
\begin{aligned}
    &\min_{\{(P_\alpha, K_{\alpha})\}_{\alpha\in\calS_l}} \min_{\alpha \in \calS_l} x_0^\top P_\alpha x_0 \\
    \text{s.t. } &P_\alpha \succeq (Q + K_\alpha^\top R K_\alpha) + (A_i + B_iK_\alpha)^\top P_\beta (A_i + B_iK_\alpha), \quad \forall\,(\alpha, \beta, i)\in\calE_l,
\end{aligned}
\label{eq:pb_P_l}
\end{equation}
decreases monotonically as the order $l$ of the De~Bruijn graph $\calH_l(M) = (\calS_l,\calE_l)$ increases. 
We denote the problem~\eqref{eq:pb_P_l} by $\mathcal P_l$, and its optimal objective value by $g_l$. 
We aim to prove that $g_l \ge g_{l+1}$ for all $l \ge 0$. 

\vspace{0.3em}
\noindent\textbf{Step 1. Construction of a candidate solution for \(\mathcal P_{l+1}\).}  
Let $\{(P_\alpha^l, K_\alpha^l)\}_{\alpha\in\calS_l}$ be an optimal solution to $\mathcal P_l$. 
Define the mapping
\[
\pi : \mathcal S_{l+1} \to \mathcal S_l,
\]
which associates to each sequence of length $l+1$ its first $l$ elements 
(i.e., the sequence obtained by removing its last switching mode). 

For any node $\alpha = (j_1,\ldots,j_{l+1}) \in \mathcal S_{l+1}$, define
\begin{equation}
    P_\alpha^{l+1} \coloneqq P^l_{\pi(\alpha)}, 
    \qquad 
    K_\alpha^{l+1} \coloneqq K^l_{\pi(\alpha)}.
\end{equation}

\vspace{0.3em}
\noindent\textbf{Step 2. Feasibility of the constructed solution.}  
We now verify that $\{(P_\alpha^{l+1}, K_\alpha^{l+1})\}_{\alpha\in\mathcal{S}_{l+1}}$ constructed above satisfies the constraints of $\mathcal P_{l+1}$.  
Let $(\alpha,\beta,j)\in\mathcal E_{l+1}$ and write $\alpha=(i_1,\ldots,i_{l+1})$.  
By definition of the De~Bruijn graph (see~\cref{def:debruijn}), we have $\beta = (j, i_1, \ldots, i_l)$.  
Therefore,
\[
\pi(\alpha) = (i_1,\ldots,i_l)
\quad\text{and}\quad
\pi(\beta) = (j,i_1,\ldots,i_{l-1}),
\]
which implies $(\pi(\alpha), \pi(\beta), j)\in\mathcal E_l$.  
Since $\{(P_\alpha^l, K_\alpha^l)\}_{\alpha\in\calS_l}$ is feasible for $\mathcal P_l$, it satisfies all corresponding LMIs.  
By construction of $\{(P_\alpha^{l+1}, K_\alpha^{l+1})\}_{\alpha \in \calS_{l+1}}$, the same inequalities therefore hold for $\mathcal P_{l+1}$, proving feasibility.

\vspace{0.3em}
\noindent\textbf{Step 3. Comparison of objective values.}  
The objective function corresponding to the constructed feasible point is
\[
\min_{\alpha \in \mathcal S_{l+1}} x_0^\top P_\alpha^{l+1} x_0
= \min_{\alpha \in \mathcal S_{l+1}} x_0^\top P_{\pi(\alpha)}^{l} x_0
= \min_{\alpha \in \mathcal S_l} x_0^\top P_\alpha^{l} x_0
= g_l.
\]
Therefore, we constructed a feasible solution $\{(P_\alpha^{l+1}, K_\alpha^{l+1})\}_{\alpha\in\mathcal{S}_{l+1}}$ to $\mathcal P_{l+1}$ with the same cost value as the optimal cost of $\mathcal{P}_l$. Since $\mathcal P_{l+1}$ minimizes the same cost over $\{(P_\alpha, K_\alpha)\}_{\alpha\in\mathcal{S}_{l+1}}$, we conclude that
\[
g_{l+1} \le g_l,
\]
which establishes the desired monotonicity.

\end{document}